\documentclass[a4paper,11pt]{amsart}

\usepackage[mac]{inputenc}
 \usepackage[T1]{fontenc}
 \usepackage[normalem]{ulem}
 \usepackage[english]{babel}
\usepackage{amsmath}
  \usepackage{amsthm}
 \usepackage{bbm}
 \usepackage{amssymb}
 \usepackage{verbatim}
 \usepackage{graphicx,tikz}
 \usepackage{color}
 \usepackage{epsfig}
 \usepackage[only,llbracket,rrbracket]{stmaryrd}
 \usepackage{enumerate}
\usepackage{setspace}

\newtheorem{thm}{Theorem}[section]

\newtheorem{lem}{Lemma}[section]
\newtheorem{coro}{Corollary}[section]
\numberwithin{equation}{section}

\theoremstyle{theorem}
{\vskip 0.5cm}

\newtheorem{rque}{\textbf{Remark}}[section]{\vskip 0.5cm} 
\newtheorem*{ack}{\textbf{Acknowledgements}}{\vskip 0.5cm}

\newcommand{\T}{{\mathbb T}}

\newcommand{\R}{{\mathbb R}}

\newcommand{\pa}{{\partial}}
\newcommand{\na}{{\nabla}}

\newcommand{\eps}{{\varepsilon}}
\newcommand{\Hc}{\mathcal{H}}
\newcommand{\Oc}{\mathcal{O}}

\renewcommand{\a}{\alpha}

\title[From Vlasov-Poisson to KdV and ZK]{From Vlasov-Poisson to Korteweg-de Vries and Zakharov-Kuznetsov}
\author{Daniel Han-Kwan}\address{\'Ecole Normale Sup\'erieure, D\'epartement de Math\'ematiques et Applications,  45 rue d'Ulm, 75230 Paris Cedex 05 France}\email{daniel.han-kwan@ens.fr}
\date{\today}
\begin{document}
 \maketitle
 
 \begin{abstract}
 We introduce a long wave scaling for the Vlasov-Poisson equation and derive, in the cold ions limit, the Korteweg-De Vries equation (in 1D) 
  and the Zakharov-Kuznetsov equation (in higher dimensions, in the presence of an external magnetic field).
 The proofs are based on the relative entropy method.
 \end{abstract}

\section{The long wave scaling of the Vlasov-Poisson equation}



\subsection{The Vlasov-Poisson system for ions with small mass electrons}

We consider the Vlasov-Poisson system which describes the evolution of ions in a plasma. 
We assume that due to their small mass, electrons in the plasma instantaneously reach their thermodynamic equilibrium, so that their density $n_e$ follows the Maxwell-Boltzmann law:
\begin{equation}
n_e = e^{{\phi}/{T_e}},
\end{equation}
denoting by $\phi$ the electric potential and $T_e>0$ the scaled temperature of electrons, taken equal to $1$ in the following. We then obtain the Vlasov-Poisson system for ions (where $t\geq0, x \in \T^d \text{  or  }  \R^d, v \in \R^d$, $d=1,2,3$):
\begin{equation}
\label{Vlasov}
\left\{
    \begin{aligned}
&  \partial_t f + v \cdot \nabla_x f + \left(E + v \wedge b\right) \cdot \nabla_v f =0,  \\
&  E = -\nabla_x \phi, \\
&  -\Delta_x \phi +  e^{\phi } = \int_{\R^d}  f \, dv,\\
&    f_{\vert t=0} =f_0.\\
    \end{aligned}
  \right.
\end{equation}
We refer to \cite{HK11} for a more complete discussion on this system. In these equations, $f(t,x,v)$ stands for the \emph{distribution function} of the ions, which allows to describe their statistical distribution: this means that $f(t,x,v) \, dx \,dv$ expresses the density of ions, at time $t$, whose position is close to $x$ and whose velocity is close to $v$. Furthermore, with usual notations, $E$ is the electric field (generated by electrons and ions themselves), while $b$ is a fixed external magnetic field.

It is interesting to notice that for a very particular class of (singular) data, namely Dirac functions in velocity, which we shall call \emph{monokinetic} data,
\begin{equation}
\label{mono}
f(t,x,v) = \rho(t,x) \delta_{v = u(t,x)},
\end{equation}
where $\delta$ denotes the Dirac delta function, we get that $f$ is a solution in the sense of distributions to \eqref{Vlasov} (with some relevant electric field) if and only if $(\rho,u)$ satisfies the following hydrodynamic system, which corresponds to the \emph{pressureless} Euler-Poisson system: 
\begin{equation}
\label{EP}
\left\{
    \begin{aligned}
&  \partial_t \rho + \na_x \cdot ( \rho u) =0,  \\
& \pa_t u + u \cdot \na_x u = E + u \wedge b,  \quad u=(u_1,u_2,u_3),\\
&  E = -\nabla_x \phi, \\
&  -\Delta_x \phi +  e^{\phi } = \rho.\\
    \end{aligned}
  \right.
\end{equation}
From the physical point of view, this corresponds to the \emph{cold ions} assumption, that corresponds to the situation where ions have zero (kinetic) temperature: any function of the form \eqref{mono} indeed satisfies
\begin{equation}
\int f(t,x,v) |v - u(t,x)|^2 \, dv =0.
\end{equation}

We also mention that a standard model in plasma physics is obtained after linearizing the Maxwell-Boltzmann law, which yields from \eqref{Vlasov} the following system:
\begin{equation}
\left\{
    \begin{aligned}
&  \partial_t f + v \cdot \nabla_x f + E \cdot \nabla_v f =0,  \\
&  E = -\nabla_x \phi, \\
&  -\Delta_x \phi +  \phi = \int_{\R^d}  f \, dv -1,\\
&    f_{\vert t=0} =f_0.\\
    \end{aligned}
  \right.
\end{equation}
This linearization produces some considerable simplifications in the analysis.

\subsection{The long wave scaling}

In some recent independent works, Lannes, Linares and Saut \cite{LLS} in the first hand, and Guo and Pu \cite{GP,PU} in the other hand, have rigorously studied the \emph{long wave limit} of the pressureless Euler-Poisson system. Precisely, one looks for solutions to \eqref{EP} under the form:
\begin{equation}
\label{ansatz}
\left\{
\begin{aligned}
& \rho=1+  \eps \rho_1 (\eps^{1/2}(x_1-t), \eps^{1/2} x_2 , \eps^{1/2} x_3, \eps^{3/2} t) + \eps^2 \rho_2 + ..., \\
& \phi= \eps \phi_1 (\eps^{1/2}(x_1-t), \eps^{1/2} x_2 , \eps^{1/2} x_3, \eps^{3/2} t) + \eps^2 \phi_2 + ..., \\
& u_1= \eps^{1/2} u_1^{(1)} (\eps^{1/2}(x_1-t), \eps^{1/2} x_2 , \eps^{1/2} x_3, \eps^{3/2} t) + \eps u_1^{(2)} + ..., \\
& u_2= \eps^{1/2} u_2^{(1)} (\eps^{1/2}(x_1-t), \eps^{1/2} x_2 , \eps^{1/2} x_3, \eps^{3/2} t) + \eps u_2^{(2)} + ..., \\
& u_3= \eps^{1/2} u_3^{(1)} (\eps^{1/2}(x_1-t), \eps^{1/2} x_2 , \eps^{1/2} x_3, \eps^{3/2} t) + \eps u_3^{(2)} + ..., \\
\end{aligned}
  \right.
\end{equation}

This leads to the study of the behaviour, as the parameter $\eps$ goes to $0$, of the solutions to the rescaled system:
\begin{equation}
\label{EPsca}
\left\{
    \begin{aligned}
&  \eps \partial_t \rho - \pa_{x_1} \rho + \na_x \cdot ( (1+ \eps \rho) u) =0,  \\
&\eps \pa_t u   -  \pa_{x_1} u+ \eps u \cdot \na_x u = E + \eps^{-1/2} u \wedge b,  \quad u=(u_1,u_2,u_3),\\
&  E = -\nabla_x \phi, \\
&  -\eps^2\Delta_x \phi +  e^{\eps\phi } = 1 + \eps \rho,\\
    \end{aligned}
  \right.
\end{equation}
In the limit, one expects to obtain some (non-linear) dispersive equations, as explained in the papers by Zakharov-Kuznetsov \cite{ZKuz} and Laedke-Spatschek \cite{LSpa}. Indeed, in $1D$, both \cite{LLS} and \cite{GP}  establish the convergence to the Korteweg-de Vries (KdV) equation, which we recall below:
\begin{equation}
 \partial_t \phi_1 + \phi_1 \pa_x \phi_1 + \pa^3_{xxx} \phi_1 = 0.
\end{equation}

In $2D$ and $3D$, which corresponds  to the setting of \eqref{EPsca},  \cite{LLS} and \cite{PU} give the derivation of a higher dimensional generalization of the KdV equation, which is referred to as the Zakharov-Kuznetsov (in short ZK) equation:
\begin{equation}
\partial_t \phi_1 +  \phi_1 \pa_x \phi_1 + \pa_{x_1}\Delta  u_1 = 0.
\end{equation}

With another kind of anisotropic (in space) scaling, Pu has also derived in $2D$ the Kadomstev-Petviashvili II  (in short KP-II) equation \cite{PU}:
\begin{equation}
 \pa_{x_1} \left( \partial_t \phi_1 +  \phi_1 \pa_{x_1} \phi_1+ \pa^3_{x_1 x_1 x_1 } \phi_1\right) + \pa^2_{x_2 x_2} \phi_1 = 0.
\end{equation}

Over the past years, there have been many mathematical studies of long wave limits towards KdV (and higher dimensional generalizations). Let us cite some works (this list is not meant to be exhaustive) concerning the following PDEs:
\begin{itemize}
\item \emph{Nonlinear Schrödinger (including Gross-Pitaevskii) equations}: Chiron and Rousset \cite{CR},  and B\'ethuel, Gravejat, Saut and Smets \cite{BGSS1,BGSS2} (with different methods),
\item \emph{General Hyperbolic systems}: Ben-Youssef and Colin \cite{BYC}, Ben-Youssef and Lannes \cite{BYL},
\item \emph{Water Waves}:  Craig \cite{CRA}, Schneider and Wayne \cite{SW}, Bona, Colin and Lannes \cite{BCL}, Alvarez-Samaniego and Lannes \cite{ASL} , Duch{e}ne \cite{DUC,DUC2},
\end{itemize}
and many others.

The fact that there exist very singular solutions to the Vlasov-Poisson system \eqref{Vlasov} which precisely yield the pressureless Euler-Poisson system \eqref{EP} suggests that is should also be possible to study \eqref{Vlasov} in a long wave regime. Following this idea, we would like to look for solutions of the form:
\begin{equation}
\left\{
\begin{aligned}
&f_\eps(t,x,v) =\eps^d \tilde f_\eps(\eps^{3/2}t, \eps^{1/2} (x_1-t),  \eps^{1/2} x_2 , \eps^{1/2} x_3, \eps^{-1} v),\\
&\phi_\eps(x,v)=  \eps \tilde\phi_\eps ( \eps^{3/2} t, \eps^{1/2}(x_1-t), \eps^{1/2} x_2 , \eps^{1/2} x_3) .
\end{aligned}
  \right.
\end{equation}
The normalization is chosen in order that the scaling of the two first moments $\rho_\eps := \int f_\eps \, dv$ and $u_\eps := \frac{\int f_\eps v \, dv}{ \int f_\eps \, dv}$ matches with the Ansatz in \eqref{ansatz}.

At some point, we will also have to somehow impose that the function $\tilde f_\eps$ is ``close'' to a Dirac function, in order to reach the Euler-Poisson dynamics.

Therefore, we propose a long wave scaling for the Vlasov-Poisson equation (in dimension $d=3$), which we introduce now:
\begin{equation}
\left\{
\begin{aligned}
&\tilde{t}= \eps^{3/2} t, \quad \tilde{x} = \eps^{1/2} x,  \quad \tilde{v} = \eps^{-1} v, \\
&\tilde{f}(\tilde{t}, \tilde{x}, \tilde{v}) = \eps^{-3} f(t,x,v), \\
&\tilde{\phi}(\tilde{t}, \tilde{x}) = \eps^{-1} \phi(t,x), \\
&\tilde{E}(\tilde{t}, \tilde{x}) = \eps^{-3/2} E(t,x).
\end{aligned}
  \right.
\end{equation}
With a slight abuse of notation (we forget the tildes), we obtain the rescaled equations: 
\begin{equation}
\left\{
    \begin{aligned}
& \eps^{3/2}   \partial_t f +  \eps^{3/2}  v \cdot \nabla_x f + \eps^{-1}  \left(\eps^{3/2}  E + \eps \, v \wedge b\right) \cdot \nabla_v f =0,  \\
&  E = - \nabla_x \phi, \\
&  -\eps^2 \Delta_x \phi +  e^{\eps \phi } =  \int_{\R^3}  f \, dv,\\
&    f_{\vert t=0} =f_0.\\
    \end{aligned}
  \right.
\end{equation}

Finally, there only remains to perform the shift with respect to the first spatial variable:
\begin{equation}
\left\{
    \begin{aligned}
&\overline{x}_1 = x_1-t, \\
&\overline{f}(t,\overline{x}_1,x_2,x_3,v) = f(t,x,v), \\
&\overline\phi(t,\overline{x}_1,x_2,x_3) = \phi(t,\overline{x}_1,x_2,x_3), \\
&\overline{E}(t,\overline{x}_1,x_2,x_3) = E(t,\overline{x}_1,x_2,x_3).
    \end{aligned}
  \right.
\end{equation}
With another abuse of notation, we end up with the rescaled Vlasov-Poisson system:
\begin{equation}
\label{syst}
\left\{
    \begin{aligned}
&  \eps \,  \partial_t f_\eps  - \pa_{x_1} f_\eps + \eps \, v  \cdot \na_{x} f_\eps + \left(E_\eps + \frac{v\wedge b}{\sqrt{\eps}} \right)\cdot  \na_v f_\eps =0,  \\
&  E_\eps = -\nabla_x \phi_\eps, \\
&  -\eps^2 \Delta_x \phi_\eps +  e^{\eps \phi_\eps} = \int_{\R^3}  f_\eps \, dv,\\
&    f_{\eps,\vert t=0} =f_{\eps,0}.\\
    \end{aligned}
  \right.
\end{equation}

To the best of our knowledge, although this scaling seems very natural, it has never been introduced yet, even in the physics literature.
Our goal in this work is to study the behaviour, as $\eps \rightarrow 0$, of solutions to \eqref{syst}, that also get close (in some sense that we shall precise later) at initial time to monokinetic data. According to the previous discussion, it is natural to expect to obtain the ZK equation in the limit.

\bigskip

The relations between the different systems are summarized in the following diagram:

\bigskip

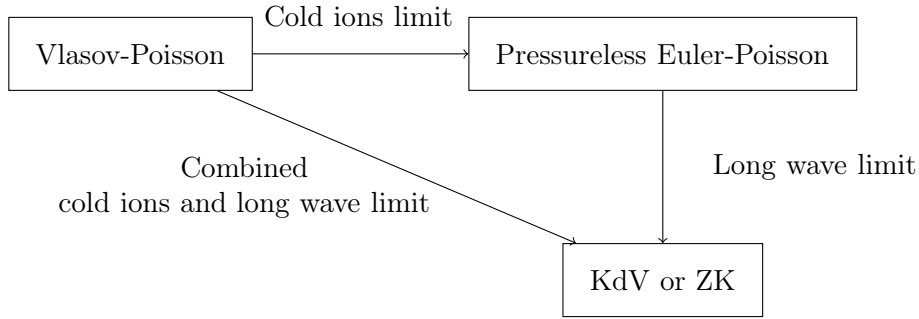
\begin{tikzpicture}
\draw
(0,0) node[inner sep=10pt,draw] (a) {Vlasov-Poisson}
(3,0.5) node {Cold ions limit}
(7,0) node[inner sep=10pt,draw] (b) {Pressureless Euler-Poisson}
(1.5,-1.5) node {Combined}
 (1.5,-2) node {cold ions and long wave limit}
(9,-1.5) node {Long wave limit}
(7,-3) node[inner sep=10pt,draw] (c) {KdV or ZK};
\draw[->] (a) -- (b);
\draw[->] (a) -- (c);
\draw[->] (b) -- (c);
\end{tikzpicture}

\bigskip

Let us briefly comment on this picture. As already explained, the link between Vlasov-Poisson and Pressureless Euler-Poisson is given by monokinetic type of data. Unfortunately, these data are way too singular to fit in the known Cauchy theories (although some results for measure data are actually available in $1D$, we refer to the work of Zheng and Majda \cite{ZM}). As a matter of fact, we are not even aware of any result proving rigorously the convergence to the pressureless Euler-Poisson system for data that would converge in some sense to monokinetic data; the stability estimates that would be needed are indeed missing.

As already said, the long wave limit from the Pressureless Euler-Poisson system has been performed in \cite{LLS, GP, PU}.
One important step is to build a solution in an interval of time which is independent of $\eps$. In order to study the long wave limit of the Vlasov-Poisson equation, instead of trying to derive the Pressureless Euler-Poisson system, the idea is to perform \emph{simultaneously} the cold ions and long wave limits. We shall start from global weak solutions to the Vlasov-Poisson equation, and therefore we will not have to face the difficulty of finding uniform lifespans. 

To prove such a result, we shall rely on a classical energy method, namely the relative entropy method. The idea originates in the work of Yau  \cite{Yau} on the hydrodynamic limit of some Ginzburg-Landau equation. It was independently brought in kinetic theory by Golse in \cite{BGP} (in the context of the incompressible Euler limit of the Boltzmann equation) and by Brenier in \cite{Br00} (in the context of the quasineutral limit of the Vlasov-Poisson equation).

The basic principle of the relative entropy strategy consists in \emph{modulating} some well-chosen functional that has to be conserved or dissipated by the physical system (for instance, the good choice is the entropy for the Boltzmann equation). The modulation is obtained in terms of the solution to the target equation. One has to ensure that this new functional allows to ``measure'' in a certain sense the distance between the solution to the original system and that to the target equation.
Then one has to prove that the functional that has been constructed is a Lyapunov function for the system: this follows from exact computations and algebraic identities. The computations can be more or less lengthy and tedious. It can be also very technical to justify these.

The estimates which can be obtained have to be understood as \emph{stability} estimates: for instance, the results proved in this work can be interpreted as the stability of monokinetic data, in the long wave regime, with a dynamics dictated by some the KdV or ZK equations. This also strongly suggests that for our long wave limit, there are stability phenomena (a la Penrose) at stake, exactly like for the case of the quasineutral limit (the effects of instabilities for the latter limit are briefly discussed for instance in \cite{Gr}); more precisely, two stream instabilities are bound to destabilize the system and make the long wave limit fail (but these are avoided when one considers monokinetic data). On the topic of instabilities for the Vlasov-Poisson system, we refer to \cite{GS1,GS2,GS3,Lin1,Lin2,MV,LZ1,LZ2}.

For the Vlasov equation, this method is remarkably well adapted to handle the cold ions limit, in other words the ``convergence'' to monokinetic data, as it will be clear later. This was observed for the study of the quasineutral limit of the Vlasov-Poisson system for electrons with fixed ions, as done by Brenier \cite{Br00} (this was completed later by Masmoudi \cite{Mas}, see also \cite{GSR}). More recently, in \cite{HK11}, we have studied the quasineutral limit of the Vlasov-Poisson for ions with small mass electrons (which corresponds to \eqref{Vlasov}).
In that work, we have observed that this equation displays  a $L \log L$ structure that is reminiscent of that of the Boltzmann equation (we refer to the works of Saint-Raymond \cite{SR1,SR2}, \cite{SRbook} for the incompressible Euler limit, see also the recent paper of Allemand \cite{All}); this will also play a crucial role in this work.

It is worth noticing that the method provided in this paper can also be applied to study the KdV limit of the Euler-Poisson system, for data with only low uniform regularity. Indeed we can start from the global weak solutions built by Cordier and Peng \cite{CP} and use similar computations as in the present paper. 

To conclude this introduction, let us mention that Haragus, Nicholls and Sattinger in \cite{HNS} relied on the KdV approximation of the Euler-Poisson system to study (formally and numerically) the interaction of solitary waves. It would be very interesting to start an analogous program for the Vlasov-Poisson equation.


\subsection{Organization of the paper} 
The paper is organized as follows: in Section \ref{sec2}, we provide the derivation of the KdV equation (see Theorem \ref{KdV}), starting from the $1D$ Vlasov-Poisson equation with a linearized Maxwell-Boltzmann law. The exposure of this relatively simple case will allow us to lay down the basic principles of the relative entropy method applied to the long wave limit. In addition to this pedagogical interest, the existence of global solutions to the KdV equation allows to give stability estimates which are valid for all times. In Section \ref{sec3}, we present the main result of this paper, which is the derivation of the ZK equation, in Theorem \ref{ZK}, starting from the $3D$ Vlasov-Poisson equation with the full Maxwell-Boltzmann law. The proof will be much more technical, in particular due to the fact that only a $L \log L$ type of control is available for the electric potential (instead of a $L^2$ bound which can be obtained with a linearized Maxwell-Boltzmann law). We will need an unusually large number of correctors in the relative entropy. Finally, we give in two appendices some variants of our  results (which can still be obtained with the relative entropy method): in particular we present another scaling for the $2D$ Vlasov-Poisson system which yields the KP-II equation in the long wave limit.

\section{From the Vlasov-Poisson equation to the Korteweg-de Vries equation}
\label{sec2}

In this section, we shall study the long wave limit of the $1D$ Vlasov-Poisson system with a linearized Maxwell-Boltzmann law, that is
(here $(x,v) \in \T \times \R$):
\begin{equation}
\label{VP1D}
\left\{
    \begin{aligned}
&  \eps \,  \partial_t f_\eps  - \pa_x f_\eps + \eps \, v  \pa_x f_\eps + E_\eps  \pa_v f_\eps =0,  \\
&  E_\eps = - \pa_x \phi_\eps, \\
&  -\eps^2 \pa^2_{xx} \phi_\eps +  \eps \phi_\eps = \int_{\R}  f_\eps \, dv -1,\\
&    f_{\eps, \vert t=0} =f_{\eps,0}.\\
    \end{aligned}
  \right.
\end{equation}

\subsection{Preliminaries} 
This system possesses an energy, which is conserved, at least formally:

\begin{equation}\begin{aligned}
\mathcal{E}_\eps (t) &:= \frac{1}{2} \int f_\eps | v |^2   \, dv \, dx + \frac{1}{2} \eps \int | \pa_x \phi_\eps|^2 \,dx + \frac{1}{2}  \int | \phi_\eps|^2 \, dx.
\end{aligned}
\end{equation}

Using this energy, as well as the conservation of $L^p_{x,v}$ norms that can be obtained using the hamiltonian structure of the Vlasov equation, one can prove, following the work of Arsenev \cite{Ar}, the following theorem, which states the existence of global weak solutions to \eqref{VP1D}:
\begin{thm}
\label{Cauchy1D}
Let $\eps>0$. Let $f_{\eps,0} \in L^1\cap L^\infty(\T \times \R)$ be a non-negative function such that the initial energy is bounded:
\begin{equation}
\mathcal{E}_\eps (0) := \frac{1}{2} \int f_{\eps,0}| v |^2   \, dv \, dx + \frac{1}{2} \eps \int | \pa_x \phi_{\eps,0}|^2 \,dx + \frac{1}{2}  \int | \phi_{\eps,0}|^2 \, dx < + \infty,
\end{equation}
where the initial electric potential $\phi_{\eps,0}$ is given by the elliptic equation:
$$
-\eps^2 \pa^2_{xx} \phi_{\eps,0} +  \eps \phi_{\eps,0} = \int_{\R}  f_{\eps,0} \, dv -1 .
$$
We also assume that:
$$
\int f_{\eps,0} \, dv \, dx =1.
$$
Then there exists a non-negative global weak solution $f_\epsilon \in  L^\infty_t(L^1\cap L^\infty(\T \times \R))$ to \eqref{VP1D}, such that the energy is non-increasing:
\begin{equation}
\forall t \geq t', \quad \mathcal{E}_\eps(t) \leq \mathcal{E}_\eps(t'),
\end{equation}
and such that the following local conservation laws for $(\rho_\eps:=\int f_\eps \, dv, J_\eps := \int f_\eps v \, dv)$ are satisfied: 
\begin{equation}
\label{LC1}
\partial_t \rho_\eps  - \frac{1}{\eps} \pa_x \rho_\eps + \pa_x  J_\eps=0,
\end{equation}
\begin{multline}
\label{LC2}
\partial_t J_\eps   - \frac{1}{\eps} \pa_x J_\eps +  \pa_x \left(\int |v|^2 f_\epsilon \, dv \right) =  
 -\frac{1}{2} \pa_x (\phi_\eps+1)^2 
 +  \frac{\eps}{2}  \pa_x \vert \pa_x V_\eps \vert^2.
\end{multline}
\end{thm}

For the KdV equation, that we recall below,
\begin{equation}
 \partial_t \phi_1 + \phi_1 \pa_x \phi_1 + \pa^3_{xxx} \phi_1 = 0,
\end{equation}
we have in hand a famous global existence result which was proved in the seminal paper of Bourgain \cite{Bour2}:
\begin{thm}
\label{KdVcauchy}
The KdV equation is globally well-posed in $H^s(\T)$ for $s\geq 0$.
\end{thm}
We will use this result only for large values of $s$.

\subsection{Formal derivation}

It is quite enlightening to perform a formal analysis in this simple one-dimensional case, to understand how the KdV equation arises.
Formally, one directly considers monokinetic data:
\begin{equation}
f_\eps (t,x,v) = \rho_\eps(t,x) \delta_{v = u_\eps(t,x)}.
\end{equation}
Then, as already explained, we obtain the following pressureless Euler-Poisson equation in a long wave scaling:
\begin{equation}
\left\{
    \begin{aligned}
&    \partial_t \rho_\eps  - \frac{1}{\eps} \pa_x \rho_\eps +  \pa_x (\rho_\eps u_\eps) =0,  \\
&    \partial_t u_\eps  - \frac{1}{\eps} \pa_x u_\eps +  u_\eps \pa_x u_\eps = \frac{1}{\eps} E_\eps , \\
&  E = - \pa_x \phi_\eps, \\
&  -\eps^2 \pa^2_{xx} \phi_\eps +  \eps \phi_\eps = \rho_\eps -1,\\
    \end{aligned}
  \right.
\end{equation}
We look for an approximate solution satisfying the Ansatz:
\begin{equation}
\left\{
\begin{aligned}
&\rho_\eps = 1  + \eps \rho_1 + \mathcal{O}(\eps^2), \\
&\phi_\eps = \phi_1 + \eps \phi_2 + \mathcal{O}(\eps^2), \\
&u_\eps = u_1 + \eps u_2 + \mathcal{O}(\eps^2).
\end{aligned}
  \right.
\end{equation}
Plugging this Ansatz in the equations, and matching the different powers of $\eps$, we obtain a cascade of equations.
\begin{itemize}
\item Conservation of charge equation:
\begin{align}
\label {a1}& \mathcal{O}(1): \quad   - \pa_{x} \rho_1 + \pa_{x} u_1 =0. \\
\label{a5} & \mathcal{O}(\eps): \quad \pa_{t} \rho_1 - \pa_x \rho_2 + \pa_x (\rho_1 u_1) + \pa_x u_2 =0.
\end{align}

\item Momentum equation:
\begin{align}
\label {a2}& \mathcal{O}(\eps^{-1}): \quad -\pa_x u_1  + \pa_x \phi_1 = 0,\\
\label{a3} & \mathcal{O}(1): \quad  \pa_t u_1 - \pa_{x} u_2 + u_1 \pa_x u_1 + \pa_x \phi_2=0.
\end{align}

\item Poisson equation
\begin{align}
 &\mathcal{O}(\eps): \quad \phi_1 = \rho_1, \\
\label{a4} & \mathcal{O}(\eps^2): \quad  -\pa^2_{xx} \phi_1 + \phi_2 = \rho_2.
\end{align}

\end{itemize}

We clearly get from \eqref{a1} and \eqref{a2}:
\begin{equation}
\phi_1 = \rho_1 = u_1
\end{equation}
and in the other hand from \eqref{a3} and \eqref{a4}:
\begin{equation}
\label{a6}
\begin{aligned}
\pa_x (\phi_2 -u_2) &= \pa_x (\phi_2 -\rho_2) + \pa_x (\rho_2 -u_2) \\
&= \pa^2_{xx} \phi_1 + \pa_t u_1 + 2 u_1 \pa_x u_1.
\end{aligned}
\end{equation}

Hence, from \eqref{a5} and \eqref{a6}, we conclude that $\phi_1$ satisfies the KdV equation:
\begin{equation}
2 \partial_t \phi_1 +3 \phi_1 \pa_x \phi_1 + \pa^3_{xxx} \phi_1 = 0.
\end{equation}

\subsection{Rigorous derivation} 

We now state the theorem which rigorously proves the convergence to KdV.

\begin{thm}
\label{KdV}
Let $(f_{\eps,0})_{\eps \in (0,1)}$ be a family of non-negative initial data satisfying the assumptions of Theorem \ref{Cauchy1D} and such that there exists $C>0$ with:
\begin{equation}
 \mathcal{E}_\eps(0) \leq C, \, \, \forall \eps \in (0,1).
\end{equation}

We denote by $(f_\eps)_{\eps \in (0,1)}$ a family of non-negative global weak solutions to \eqref{VP1D} given by Theorem \ref{Cauchy1D}.
Let $\mathcal{H}_\eps$ be the relative entropy defined by the functional:
\begin{equation}
\label{RE}
\begin{aligned}
\mathcal{H}_\eps (t) &:= \frac{1}{2} \int f_\eps | v - u_1 - \eps u_2 |^2 \, dv \, dx + \frac{1}{2} \eps \int | \pa_x \phi_\eps - \pa_x \phi_1 - \eps \pa_x \phi_2 |^2 \,dx \\
&+ \frac{1}{2} \int  (\phi_\eps - \phi_1 - \eps \phi_2)^2 \, dx,
\end{aligned}
\end{equation}
where $(u_1,\phi_1,u_2, \phi_2) \in [C([0,+\infty[,H^{s+2}(\T))]^4$ (with $s>3/2$) satisfy the following system:
\begin{equation}
\label{iden} 
\left\{
    \begin{aligned}
& 2 \partial_t \phi_1 +3 \phi_1 \pa_x \phi_1 + \pa^3_{xxx} \phi_1 = 0,  \\
&   u_1 = \phi_1, \\
&  \pa_x (u_2 - \phi_2) = \pa_t \phi_1 + \phi_1 \pa_x \phi_1. 
    \end{aligned}
  \right.
\end{equation}

Then there exist $C_1,C_2>0$, such that for all $\eps\in (0,1)$, 
\begin{equation}
\forall t \geq 0,  \quad \mathcal{H}_\eps (t) \leq \Hc_\eps(0)+  \int_0^t  (C_1 \mathcal{H}_\eps (s)  + C_2 \sqrt{\eps}) \, ds.
\end{equation}

Assuming in addition that there exists $C_3$ such that for any $\eps \in (0,1)$:
\begin{equation}
\mathcal{H}_\eps (0) \leq C_3  \sqrt{\eps},
\end{equation}
then we obtain for any $\eps \in (0,1)$:
\begin{equation}
\label{gron}
\forall t \geq 0, \quad \mathcal{H}_\eps (t) \leq C_3   e^{C_1 t}  \sqrt{\eps} + C_2 \frac{e^{C_1 t}-1}{C_1} \sqrt{\eps}.
\end{equation}

\end{thm}

In order to get functions satisfying \eqref{iden}, we can proceed as follows:
\begin{itemize}
\item The existence of $\phi_1$ is ensured by Theorem \ref{KdVcauchy}.
\item We accordingly set $\phi_1 =u_1$.
\item The functions $\phi_2$ and $u_2$ can be seen as correctors. The last equation of \eqref{iden} allows to define them. Only the value of the function $u_2 - \phi_2$ is important.
\end{itemize}

\begin{rque}
The estimate \eqref{gron} clearly implies that

\begin{equation}
\label{uno}
\frac{1}{2} \int f_\eps | v - u_1 - \eps u_2 |^2 \, dv \, dx  \leq C_3   e^{C_1 t}  \sqrt{\eps} + C_2 \frac{e^{C_1 t}-1}{C_1} \sqrt{\eps},
\end{equation}
and
\begin{equation}
\label{duo}
\frac{1}{2} \int  (\phi_\eps - \phi_1 - \eps \phi_2)^2 \, dx  \leq  C_3   e^{C_1 t}  \sqrt{\eps} + C_2 \frac{e^{C_1 t}-1}{C_1} \sqrt{\eps}.
\end{equation}

From \eqref{uno}, denoting by $f$ a weak limit of $f_\eps$, we deduce that necessarily
\begin{equation}
 \int f | v - u_1 |^2 \, dv \, dx =0,
\end{equation}
which means that the limit temperature is zero (cold ions limit).

\end{rque}

\begin{rque}
The estimate \eqref{gron} is valid for all times, but it is useful for times of order ${o}(|\log \eps|)$ (in other words, logarithmically growing times). This is slightly better than the times of validity obtained in the long wave limit of Euler-Poisson in \cite{LLS}, which are (translated in our framework) of order $O(1)$.
\end{rque}

From this theorem, we can deduce the following corollary:
\begin{coro}
\label{coro1}
Making the same assumptions as in the previous theorem, we obtain the weak convergences:
\begin{equation}
\left\{
\begin{aligned}
&\rho_\eps \rightharpoonup_{\eps\rightarrow 0} 1 \text{  in  } L^\infty_t \mathcal{M}^1 \text{  weak-*}, \\
&J_\eps \rightharpoonup_{\eps\rightarrow 0} \phi_1 \text{  in  } L^\infty_t \mathcal{M}^1 \text{  weak-*}, \\
&\phi_\eps \rightharpoonup_{\eps\rightarrow 0} \phi_1 \text{  in  } L^\infty_t \mathcal{M}^1 \text{  weak-*}.
\end{aligned}
\right.
\end{equation}
\end{coro}

\begin{proof}[Proof of Corollary \ref{coro1}]

By conservation of the $L^1$ norm, we deduce that
$$
\| \rho_\eps \|_{L^1} \leq C.
$$
Using this bound and the one coming from the energy inequality:
$$
\int f_\eps |v|^2 \, dv \, dx \leq C,
$$
and by non-negativity of $f_\eps$, it is standard to deduce a $L^1$ control on $J_\eps$: 
$$
\| \rho_\eps \|_{L^1} \leq C.
$$

Therefore, there exist two non-negative measures $\rho, J$ such that up to some extraction, we have:
\begin{equation}
\left\{
\begin{aligned}
&\rho_\eps \rightharpoonup \rho, \\
&J_\eps \rightharpoonup J,
\end{aligned}
\right.
\end{equation}
in the vague sense of measures. We have to show that $\rho=1$ and $J= \phi_1$.

We pass to the limit $\eps \rightarrow 0$ (in the sense of distributions) in the Poisson equation:
$$
-\eps^2 \pa^2_{xx} \phi_\eps +  \eps \phi_\eps = \rho_\eps -1,
$$
using the uniform $L^2$ bound on $\sqrt \eps \pa_x \phi_\eps$ and $\phi_\eps$ (coming from the energy inequality), we deduce that we necessarily have $\rho=1$, which proves the first claim.

In the other hand, using the Cauchy-Schwarz inequality, we have
\begin{equation}
\begin{aligned}
\frac{\vert J_{\epsilon} - \rho_\epsilon (u_1 + \eps u_2)\vert^2}{\rho_\epsilon}= \frac{\left(\int f_\epsilon( v - u_1 - \eps u_2)dv\right)^2}{\int f_\epsilon dv}\leq  \int f_\epsilon\vert  v - u_1 - \eps u_2\vert^2 \, dv,
\end{aligned}
\end{equation}

The functional $(\rho,J) \rightarrow \int \frac{\vert J - \rho (u_1 + \eps u_2)\vert^2}{\rho} \, dx$ is convex and lower semi-continuous with respect to the weak convergence of measures (see \cite{Br00}). As a consequence, the weak convergences in the vague sense of measures $\rho_\epsilon \rightharpoonup 1$ and $J_\epsilon \rightharpoonup J$ lead to:

\begin{equation}
\int  {\vert J -  u_1 \vert^2} dx \leq \liminf_{\epsilon\rightarrow 0} \int \frac{\vert J_\epsilon - \rho_\epsilon u\vert^2}{\rho_\epsilon}dx.
\end{equation}

By \eqref{gron}, we deduce that $J= u_1 (= \phi_1)$.

To conclude, the uniqueness of the limit allows us to say that the weak convergences actually hold without any extraction.

\end{proof}

\begin{rque}
We can actually state strong convergence results. Indeed, in view of the preceeding proof, it is clear that \eqref{uno} implies the ``strong'' convergence:
\begin{equation}
\left\| \frac{ J_{\epsilon} - \rho_\epsilon (u_1 + \eps u_2)}{\sqrt \rho_\epsilon} \right\|_{L^2}^2 \leq  C_3   e^{C_1 t}  \sqrt{\eps} + C_2 \frac{e^{C_1 t}-1}{C_1} \sqrt{\eps}.
\end{equation}

In the other hand, the control \eqref{duo} means that $\phi_\eps$ converges strongly in $L^2$ to $\phi_1$, as $\eps$ goes to $0$.
\end{rque}

\subsection{Proof of Theorem \ref{KdV}}
Relying on the fact the energy is non-increasing (and that $\Hc_\eps$ is built as a modulation of the energy), we have for all $t\geq0$ and all $\eps \in (0,1)$:
$$
\Hc_\eps(t) = \mathcal{E}_\eps(t) + (\Hc_\eps(t)- \mathcal{E}_\eps(t)) \leq \mathcal{E}_\eps(0) +  (\Hc_\eps(t)- \mathcal{E}_\eps(t)),
$$
which yields:
\begin{align*}
\Hc_\eps (t) &\leq \Hc_\eps(0) + \int_0^t \int \pa_t \left[ f_\eps \left( \frac{1}{2} |u_1 + \eps u_2 |^2 - v (u_1 + \eps u_2) \right) \right] \, dv \, dx \\
&+ \eps \int \pa_t \left[ \frac{1}{2} |\pa_x \phi_1 + \eps \pa_x \phi_2 |^2 - \pa_x \phi_\eps (\pa_x \phi_1 + \eps \pa_x \phi_2)\right ] \, dx \\
&+ \int \pa_t \left [\frac{1}{2}| \phi_1 + \eps \phi_2 |^2 - \phi_\eps ( \phi_1 + \eps  \phi_2) \right ] \, dx \, ds \\
&:=\Hc_\eps(0) + \int_0^t  (I_1 + I_2 + I_3) \, ds.
\end{align*}

We are now going to study $I_1$, $I_2$ and $I_3$. 
The computations can be justified using only the local conservation laws \eqref{LC1} and \eqref{LC2}.

\medskip
{\bf Study of $I_1$}.

Using the fact that $f_\eps$ satisfies the Vlasov-Poisson equation, we obtain the identity:
\begin{align*}
&\int \pa_t  f_\eps \left( \frac{1}{2} |u_1 + \eps u_2 |^2 - v (u_1 + \eps u_2) \right) \, dv \, dx \\
&= \int f_\eps (u_1 + \eps u_2 - v ) \Big[ - \frac{1}{\eps} \pa_x u_1   + v \pa_x u_1 - \pa_x u_2  
+ \eps  v \pa_x u_2 \Big]  \, dv \, dx \\
&- \int \frac{1}{\eps} \rho_\eps E_\eps u_1  \, dx  - \int  \rho_\eps E_\eps  u_2 \, dx.
\end{align*}

In order to obtain an hydrodynamic equation inside $[...]$, we write:
\begin{multline*}
 \int f_\eps (u_1 + \eps u_2 - v )  (v \pa_x u_1) =   \int f_\eps (u_1 + \eps u_2 - v )  (u_1 \pa_x u_1)  \\
- \int f_\eps |u_1 + \eps u_2 - v|^2  \pa_x u_1 + \int \eps f_\eps (u_1 + \eps u_2 - v ) u_2 \pa_x u_1.
\end{multline*}

After deriving with respect to time the term $ \frac{1}{2} |u_1 + \eps u_2 |^2 - v (u_1 + \eps u_2)$, we get the following contribution in $I_1$:
$$
 \int f_\eps (u_1 + \eps u_2 - v ) \Big[  \pa_t u_1  
+ \eps \pa_t u_2 \Big]  \, dv \, dx.
$$
We now focus on the terms of order $\mathcal{O}(1/\eps)$, for which we can write
\begin{multline*}
 - \int f_\eps (u_1 + \eps u_2 - v ) \frac{1}{\eps} \pa_x u_1   \, dv \, dx \\
 =  \int f_\eps (u_1 + \eps u_2 - v ) \frac{1}{\eps} (-\pa_x u_1 + \pa_x \phi_1)   \, dv \, dx  \\
 - \int \frac{1}{\eps}  \rho_\eps  u_1 \pa_x \phi_1 \, dx   - \int   \rho_\eps  u_2 \pa_x \phi_1  \, dx +  \int \frac{1}{\eps}  J_\eps   \pa_x \phi_1  \, dx,
\end{multline*}
and
\begin{multline*}
- \int \frac{1}{\eps} \rho_\eps E_\eps u_1  \, dx  - \int \frac{1}{\eps}  \rho_\eps  u_1 \pa_x \phi_1 \, dx -  \int \rho_\eps u_1 \pa_x \phi_2 \, dx \\
= - \frac{1}{\eps} \int \rho_\eps u_1 (\pa_x \phi_1 + \eps \pa_x \phi_2 - \pa_x \phi) \, dx.
\end{multline*}

In the other hand, we observe that we can write:
\begin{multline*}
 \int f_\eps (u_1 + \eps u_2 - v ) \Big[ \pa_t u_1 + u_1 \pa_x u_1 - \pa_x u_2  \Big]  \, dv \, dx \\
 =  \int f_\eps (u_1 + \eps u_2 - v ) \Big[ \pa_t u_1 + u_1 \pa_x u_1 - \pa_x u_2  + \pa_x \phi_2 \Big]  \, dv \, dx \\
 + \int J_\eps \pa_x \phi_2 \, dx  -  \int \rho_\eps u_1 \pa_x \phi_2 \, dx  - \eps \int \rho_\eps u_2 \pa_x \phi_2 \, dx.
\end{multline*}

Using the $L^1$ uniform bounds for $\rho_\eps$ and $J_\eps$, as well as the various Lipschitz bounds on $(u_1,u_2)$, it is clear that 
$$\left|\eps \int \rho_\eps u_2 \pa_x \phi_2 \, dx\right| + \left| \eps\int f_\eps (u_1 + \eps u_2 - v ) \pa_t u_2
 \, dv \, dx\right|\leq C \eps.$$

In the end, we have:
\begin{equation}
\begin{aligned}
I_1 &= \frac{1}{\eps} \int f_\eps (u_1 + \eps u_2 - v ) \Big[ - \pa_x u_1  + \pa_x \phi_1  \Big]  \, dv \, dx \\
&+ \int f_\eps (u_1 + \eps u_2 - v ) \Big[  \pa_t u_1 + v \pa_x u_1 - \pa_x u_2  + \pa_x \phi_2 \Big]\, dv \, dx \\
&+ \int J_\eps \left(\frac{1}{\eps}\pa_x \phi_1 + \pa_x \phi_2\right) \, dx \\
&- \frac{1}{\eps} \int \rho_\eps u_1 (\pa_x \phi_1 + \eps \pa_x \phi_2 - \pa_x \phi) \, dx \\
&- \frac{1}{\eps} \int \rho_\eps u_2 \pa_x( \phi_1 + \eps \phi_2)  \,dx  - \int \rho_\eps E_\eps u_2 \, dx \\
&- \int f_\eps |u_1 + \eps u_2 - v|^2  \pa_x u_1  +  \mathcal{O}(\eps),
\end{aligned}
\end{equation}
where $ \mathcal{O}(\eps)$ is a notation for all the terms which can be bounded by $C_\eps$, with $C>0$ independent of $\eps$.

\medskip
{\bf Study of $I_2$}.

We have:
\begin{multline*}
I_2 = \eps \int \pa_t (\pa_x \phi_1 + \eps \pa_x  \phi_2) (-\pa_x  \phi_\eps + \pa_x  \phi_1 + \eps \pa_x \phi_2) \, dx \\
- \eps \int \pa_t \pa_x  \phi_\eps (\pa_x  \phi_1 + \eps \pa_x  \phi_2) \, dx.
\end{multline*}

We get the easy bound (using $|ab| \leq \frac{1}{2}(a^2 + b^2)$):
\begin{align*}
&\eps \left| \int (\pa_x \phi_1 + \eps \pa_x  \phi_2)(-\pa_x  \phi_\eps + \pa_x  \phi_1 + \eps \pa_x \phi_2) \, dx\right|\\
& \leq
C \left[\eps  \int (\pa_t  (\pa_x \phi_1 + \eps \pa_x  \phi_2))^2 \, dx + \eps \int (-\pa_x  \phi_\eps + \pa_x  \phi_1 + \eps \pa_x \phi_2)^2 \, dx \right] \\
&\leq \mathcal{O}(\eps) + C \Hc_\eps(t).
\end{align*}

\medskip
{\bf Study of $I_3$}.

We get:
\begin{multline*}
I_3 = \int \pa_t (\phi_1 + \eps \phi_2) (-\phi_\eps + \phi_1 + \eps \phi_2) \, dx \\
- \int \pa_t \phi_\eps (\phi_1 + \eps \phi_2) \, dx := I_3^1 + I_3^2.
\end{multline*}

Let us start with $I_3^1$, that we can rewrite as:
\begin{multline*}
I_3^1 = \int (\phi_1 - \phi_\eps) \pa_t \phi_1 +  \eps \int \pa_t \phi_2 (-\phi_\eps + \phi_1 + \eps \phi_2) \, dx + \eps \int \pa_t \phi_1 \phi_2 \, dx. 
\end{multline*}

If we differentiate with respect to time the Poisson equation, we obtain:
\begin{align*}
-\pa_t \phi_\eps = - \eps \pa^2_{xx} \pa_t \phi_\eps - \frac{1}{\eps} \pa_t \rho_\eps.
\end{align*}
Then, using the local conservation of charge (in shifted variables) that we recall below,
\begin{align*}
-  \frac{1}{\eps} \pa_t \rho_\eps = -  \frac{1}{\eps^2}\pa_x \rho_\eps +  \frac{1}{\eps} \pa_x J_\eps,
\end{align*}
we obtain that:
\begin{multline*}
I_3^2 = - \frac{1}{\eps^2} \int \pa_x \rho_\eps (\phi_1 + \eps \phi_2) \, dx   \\
+ \frac{1}{\eps} \int \pa_x J_\eps \phi_1 \, dx + \int \pa_x J_\eps \phi_2 \, dx \\
- \eps \int \pa_t \pa_{xx}^2 \phi_\eps (\phi_1 + \eps \phi_2).
\end{multline*}

Of course, by integration by parts, we can rewrite the last three terms as follows: 
\begin{multline*}
 \frac{1}{\eps} \int \pa_x J_\eps \phi_1 \, dx + \int \pa_x J_\eps \phi_2 \, dx + \eps \int \pa_t \pa_{xx}^2 \phi_\eps (\phi_1 + \eps \phi_2) \\
  = - \frac{1}{\eps} \int  J_\eps \pa_x \phi_1 \, dx - \int  J_\eps \pa_x \phi_2 \, dx + \eps \int \pa_t \pa_{x} \phi_\eps (\pa_x \phi_1 + \eps \pa_x \phi_2),
\end{multline*} 
which get simplified using some terms coming from the computations of $I_1$ and $I_2$.

\medskip
{\bf Study of the remaining significant terms}.

We now have to study the following terms (coming from the computation of $I_1 + I_2 + I_3$):
\begin{align*}
K_1 &:= - \int \rho_\eps u_2 \pa_x (\phi_1 + \eps \phi_2) \, dx, \quad
K_2 := - \int \rho_\eps E_\eps u_2 \, dx, \\
K_3 &:= - \frac{1}{\eps} \int  \rho_\eps u_1 (\pa_x \phi_1 + \eps \pa_x \phi_2 - \pa_x \phi_\eps) \, dx, \\
K_4 &:= - \frac{1}{\eps^2} \int \pa_x \rho_\eps (\phi_1 + \eps \phi_2) \, dx, \\
K_5 &:= \int (\phi_1 - \phi_\eps) \pa_t \phi_1\, dx.  \\
\end{align*}

We start with $K_1$ and $K_2$. Using the Poisson equation, we have:
\begin{multline*}
K_1 = - \int u_2 \pa_x(\phi_1 + \eps \phi_2) \, dx - \eps \int \phi_\eps u_2 \pa_x (\phi_1 + \eps \phi_2)\, dx  \\
+ \eps^2 \int \pa^2_{xx} \phi_\eps u_2 \pa_x (\phi_1 + \eps \phi_2) \, dx \\
=  \int \pa_x u_2  \phi_1 \, dx - \eps \int \phi_\eps u_2 \pa_x \phi_1 \, dx + \eps^2 \int \pa^2_{xx} \phi_\eps u_2 \pa_x (\phi_1 + \eps \phi_2) \, dx \\
+  \mathcal{O}(\eps).
\end{multline*}

We have (after integration by parts),
\begin{align*}
\eps^2 \left|\int \pa^2_{xx} \phi_\eps u_2 \pa_x (\phi_1 + \eps \phi_2) \, dx \right| \leq C\eps.
\end{align*}

Similarly, we compute
\begin{multline*}
K_2 =  \int u_2 \pa_x \phi_\eps \, dx + \eps \int \phi_\eps u_2 \pa_x \phi_\eps \, dx - \eps^2 \int \pa^2_{xx} \phi_\eps u_2 \pa_x \phi_\eps \, dx \\
= - \int \pa_x u_2  \phi_\eps \, dx - \eps \int \frac{\phi_\eps^2}{2} \pa_x u_2 \ \, dx + \eps^2 \int \frac{(\pa_x \phi_\eps)^2}{2} \pa_x u_2 \, dx.
\end{multline*}

As a result, one gets
\begin{equation*}
K_1 + K_2 =  \int (\phi_1 - \phi_\eps) \pa_x u_2  \, dx + \mathcal{O}(\eps).
\end{equation*}

Concerning $K_3$, we have (once again using the Poisson equation):

\begin{align*}
K_3 &=  - \frac{1}{\eps} \int  u_1 (\pa_x \phi_1 + \eps \pa_x \phi_2 - \pa_x \phi_\eps) \, dx \\
&-  \int  \phi_\eps u_1 (\pa_x \phi_1 + \eps \pa_x \phi_2 - \pa_x \phi_\eps) \, dx \\
&+  \int  \eps \pa^2_{xx} \phi_\eps u_1 (\pa_x \phi_1 + \eps \pa_x \phi_2 - \pa_x \phi_\eps) \, dx \\
&:= K_3^1 + K_3^2 + K_3^3.
\end{align*}

For $K_3^1$, we have:
\begin{align*}
K_3^1 =- \frac{1}{\eps} \int  u_1 \pa_x \phi_1 \, dx - \int u_1 \pa_x \phi_2 \, dx  +  \frac{1}{\eps} \int  u_1 \pa_x \phi_\eps \, dx.
\end{align*}

Concerning $K_3^2$, we write:
\begin{align*}
K_3^2 = - \int \phi_\eps u_1 \pa_x u_1 \, dx + \int \phi_\eps \pa_x \phi_\eps \, u_1 \, dx +  \eps \int \phi_\eps u_1 \pa_x \phi_2 \, dx  
\end{align*}
and
\begin{equation*}
- \int \phi_\eps u_1 \pa_x u_1 \, dx = \int (\phi_1 - \phi_\eps) u_1 \pa_x u_1 \, dx - \int \phi_1 u_1 \pa_x u_1 \, dx.
\end{equation*}

The contribution coming from $ \int \phi_\eps \pa_x \phi_\eps \, u_1 \, dx$ could be harmful (a priori it is of order $\Oc(1/{\eps})$), but we can rely on the following identities in order to make the relative entropy appear:
\begin{align*}
& \int \phi_\eps \pa_x \phi_\eps \, u_1 \, dx = - \frac{1}{2} \int \phi^2_\eps \pa_x u_1 \, dx \\
& = - \frac{1}{2} \int (\phi_\eps- \phi_1 - \eps \phi_2)^2 \pa_x u_1 \, dx  - \int \phi_\eps  \phi_1 \pa_x u_1 \, dx  + \eps \int \phi_\eps \phi_2 \pa_x u_1 \, dx \\
& \quad \quad + \frac{1}{2} \int (\phi_1 + \eps \phi_2)^2 \pa_x u_1\, dx \\
& = - \frac{1}{2} \int (\phi_\eps- \phi_1 - \eps \phi_2)^2 \pa_x u_1 \, dx  + \int (\phi_1 - \phi_\eps)  \phi_1 \pa_x u_1 \, dx  + \eps \int \phi_\eps \phi_2 \pa_x u_1 \, dx \\
& \quad \quad - \frac{1}{2} \int \phi_1^2 \pa_x u_1\, dx + \eps \int \left(\phi_1 \phi_2 \pa_x u_1 + \frac{1}{2} \eps \phi_2^2 \pa_x u_1 \right)\, dx.
\end{align*}
We note that $\eps \int \phi_\eps \phi_2 \pa_x u_1 \, dx$ is of order $\eps$, using the $L^2$ uniform bound on $\phi_\eps$ granted by the energy.

Finally, for $K_3^3$, we have
\begin{align*}
K_3^3 =-  \int   \pa^2_{xx} \phi_\eps u_1  \pa_x \phi_\eps \, dx +  \eps \int   \pa^2_{xx} \phi_\eps u_1 (\pa_x \phi_1 + \eps \pa_x \phi_2) \, dx.
\end{align*}

As before, the most significant term can be rewritten as follows:
\begin{align*}
&-  \int   \pa^2_{xx} \phi_\eps u_1  \pa_x \phi_\eps \, dx = \eps \int \frac{(\pa_x \phi_\eps)^2}{2} \pa_x u_1 \, dx \\
&= \frac{1}{2} \int \eps (\pa_x \phi_\eps - \pa_x \phi_1 - \eps \pa_x \phi_2)^2 \pa_x u_1 + \eps \int \pa_x \phi_\eps (\pa_x \phi_1+ \eps \pa_x \phi_2) \pa_x u_1 \, dx \\
&- \frac{1}{2} \int \eps (\pa_x \phi_1 + \eps \pa_x \phi_2)^2 \pa_x u_1 \, dx.
\end{align*}

We have the bound:
$$
\eps \left| \int \pa_x \phi_\eps (\pa_x \phi_1+ \eps \pa_x \phi_2) \pa_x u_1 \, dx\right|  \leq C \sqrt{\eps}.
$$

We treat $K_4$ exactly as $K_3$.
\begin{align*}
K_4 &=  \frac{1}{\eps^2} \int \rho_\eps (\pa_x \phi_1 + \eps \pa_x \phi_2) \, dx \\
&=  \frac{1}{\eps^2} \int   (\pa_x \phi_1 + \eps \pa_x \phi_2 ) \, dx +  \frac{1}{\eps} \int  \phi_\eps  (\pa_x \phi_1 + \eps \pa_x \phi_2) \, dx \\
&-  \int   \pa^2_{xx} \phi_\eps  (\pa_x \phi_1 + \eps \pa_x \phi_2) \, dx \\
&:= K_4^1 + K_4^2 + K_4^3.
\end{align*}

Clearly, we have $K_4^1=0$ and
\begin{align*}
K_4^2 =  \frac{1}{\eps} \int \phi_\eps \pa_x \phi_1 \, dx +  \int \phi_\eps \pa_x \phi_2 \, dx, 
\end{align*}
and since $\int \phi_1  \pa^3_{xxx} \phi_1 \, dx=0$, we have:
$$K_4^3 = \int (\phi_1 - \phi_\eps) \pa^3_{xxx} \phi_1 \, dx  - \eps \int \pa_{xx}^2 \phi_\eps \pa_x \phi_2 \, dx.$$
By integration by parts, it is clear that the second term of $K_4^3$ is of order $\sqrt{\eps}$.

\medskip
{\bf Conclusion}.

Gathering all pieces together, we obtain that:
\begin{align*}
&\Hc_\eps(t) \leq \Hc_\eps(t) + \int_0^t \Big[  \int f_\eps (u_1 + \eps u_2 - v ) \frac{1}{\eps} (-\pa_x u_1 + \pa_x \phi_1)   \, dv \, dx  \\
&+ \int f_\eps (u_1 + \eps u_2 - v ) \Big[ \pa_t u_1 + u_1 \pa_x u_1 - \pa_x u_2  + \pa_x \phi_2 \Big]  \, dv \, dx \\
&+ \frac{1}{\eps} \int (u_1- \phi_1) \pa_x \phi_\eps \, dx \\
&+ \int (\phi_1 - \phi_\eps) \left( \pa_t \phi_1 +  \phi_1 \pa_x u_1 + u_1 \pa_x u_1 + \pa^3_{xxx} \phi_1 + \pa_x u_2 - \pa_x \phi_2) \right) \, dx \\
&- \int (u_1-\phi_1) \pa_x \phi_2  - \int \phi_1 u_1 \pa_x u_1 \, dx - \frac{1}{2} \int \phi_1^2 \pa_x u_1 \\
&-   \int f_\eps | v - u_1 - \eps u_2 |^2 \pa_x u_1 \, dv \, dx - \frac{1}{2} \eps \int | \pa_x \phi_\eps - \pa_x \phi_1 - \eps \pa_x \phi_2 |^2  \pa_x u_1 \,dx \\
&- \frac{1}{2} \int  (\phi_\eps - \phi_1 - \eps \phi_2)^2 \pa_x u_1 \, dx \\
&+ \mathcal{O}(\sqrt \epsilon)\Big] \, ds.
\end{align*}

We impose the following cancellations (to kill all singular terms):
\begin{equation}
\left\{
    \begin{aligned}
& u_1 -  \phi_1 = 0,  \\
&   \pa_t u_1 + u_1 \pa_x u_1 - \pa_x u_2  + \pa_x \phi_2=0 , \\
&  \pa_t \phi_1 +  \phi_1 \pa_x u_1 + u_1 \pa_x u_1 + \pa^3_{xxx} \phi_1 + \pa_x u_2 - \pa_x \phi_2=0.
    \end{aligned}
  \right.
\end{equation}
These are consequences of the identity \eqref{iden}. Using the Lipschitz bound on $u_1$ we end up with:
\begin{equation}
 \Hc_\eps(t) \leq    \Hc_\eps(0)  + \int_0^t (C_1 \Hc_\eps (t) + C_2 \sqrt{\eps}) \, ds,
\end{equation}
which then yields the claimed Gronwall type bound \eqref{gron}.
The proof is consequently complete.




\section{From the Vlasov-Poisson equation to the Zakharov-Kuznetsov equation} 
\label{sec3}

We perform the analysis in $3D$, but this can also be done in $2D$, in an almost similar way.
Let $(e_1,e_2,e_3)$ be an orthonormal basis of $\R^3$; for simplicity we fix $b= e_1$.

We study the behaviour, as $\eps \rightarrow 0$, of the solutions to the following equation (for $(x,v) \in \T^3 \times \R^3$):
\begin{equation}
\label{VP3D}
\left\{
    \begin{aligned}
&  \eps \,  \partial_t f_\eps  - \pa_{x_1} f_\eps + \eps \, v  \cdot \na_{x} f_\eps + \left(E_\eps + \frac{v\wedge e_1}{\sqrt{\eps}} \right)\cdot  \na_v f_\eps =0,  \\
&  E_\eps = -\nabla_x \phi_\eps, \\
&  -\eps^2 \Delta_x \phi_\eps +  e^{\eps \phi_\eps} = \int_{\R^3}  f_\eps \, dv,\\
&    f_{\eps,\vert t=0} =f_{\eps,0}.\\
    \end{aligned}
  \right.
\end{equation}

\subsection{Preliminaries}

This system possesses an energy, which is conserved, at least formally:

\begin{equation}
\begin{aligned}
\mathcal{E}_\eps (t) &:= \frac{1}{2} \int f_\eps | v |^2   \, dv \, dx \\ 
&+ \frac{1}{2} \eps \int | \na_x \phi_\eps|^2 \,dx + \frac{1}{\eps^2} \int  ( e^{\eps \phi_\eps} \left(\eps \phi_\eps -1 \right) + 1 )  \, dx, 
\end{aligned}
\end{equation}

Note that the third term of this energy has a $L \log L$ structure:
$$
\frac{1}{\eps^2} \int  ( e^{\eps \phi_\eps} \left(\eps \phi_\eps -1 \right) + 1 )  \, dx
= \frac{1}{\eps^2} \int  ( e^{\eps \phi_\eps} \log (e^{\eps \phi_\eps}/ e^{0}) - e^{\eps \phi_\eps} + e^{0})   \, dx.
$$

We have the following global existence theorem, which can be adapted from the work of Bouchut \cite{Bou}:
\begin{thm}
\label{Cauchy3D}
Let $\eps>0$. Let $f_{\eps,0} \in L^1\cap L^\infty(\T^3 \times \R^3)$ be a non-negative function such that:
\begin{equation}
\begin{aligned}
\mathcal{E}_\eps (0) &:= \frac{1}{2} \int f_{\eps,0} | v |^2   \, dv \, dx \\ 
&+ \frac{1}{2} \eps \int | \na_x \phi_{\eps,0}|^2 \,dx + \frac{1}{\eps^2} \int   (e^{\eps \phi_{\eps,0}} \left(\eps \phi_{\eps,0} -1 \right) +1)  \, dx < + \infty,
\end{aligned}
\end{equation}
where the initial electric potential $\phi_{\eps,0}$ is given by the elliptic equation:
$$
-\eps^2 \Delta_x \phi_{\eps,0} +  e^{\eps \phi_{\eps,0}} = \int_{\R^3}  f_{\eps,0} \, dv.
$$
We also assume that:
$$
\int f_{\eps,0} \, dv \, dx =1.
$$
Then there exists a non-negative global weak solution $f_\epsilon \in  L^\infty_t(L^1\cap L^\infty(\T^3 \times \R^3))$ to \eqref{VP3D}, such that:
\begin{equation}
\forall t \geq t', \quad \mathcal{E}_\eps(t) \leq \mathcal{E}_\eps(t'),
\end{equation}
and such that the following conservation laws for $(\rho_\eps:=\int f_\eps \, dv, J_\eps := \int f_\eps v \, dv)$ are satisfied: 
\begin{equation}
\label{LC1'}
\partial_t \rho_\eps - \frac{1}{\eps} \pa_{x_1} \rho_\eps + \nabla_x \cdot J_\eps=0,
\end{equation}
\begin{multline}
\label{LC2'}
\partial_t J_\eps - \frac{1}{\eps} \pa_{x_1} J_\eps +  \nabla_x : \left(\int v \otimes v f_\eps \, dv \right) = \\
  - \frac{1}{\eps} \nabla_x (e^{\eps \phi_\eps}) + \eps \na_x \cdot (\nabla_x V_\eps \otimes \nabla_x V_\eps) - \frac{\eps}{2}  \nabla_x \vert \nabla_x V_\eps \vert^2.
\end{multline}

\end{thm}

Let us now turn to the Cauchy problem for the ZK equation, that we recall now:
\begin{equation}
\label{ZKeq}
\partial_t \phi_1 +  \phi_1 \pa_x \phi_1 + \pa_{x_1}\Delta  u_1 = 0.
\end{equation}
The only result in $3D$ we are aware of for this equation concerns the case of the whole space $\R^3$ (for results in $2D$, we refer to \cite{Fam,LP}): in \cite{LS}, Linares and Saut  proved that ZK is locally well-posed in $H^s(\R^3)$, for $s>9/8$, and more recently in \cite{RV},  Ribaud and Vento showed that it is well-posed for $s>1$. Their proofs are based on dispersive effects and can not be directly applied to the case of the torus $\T^3$. Using standard methods, we can nevertheless prove the easy theorem:
\begin{thm}
\label{ZKcauchy}
The ZK equation is locally well-posed in $H^s(\T^3)$ for $s> 5/2$.
\end{thm}
This will be sufficient for our needs.

\subsection{Rigorous convergence result}
Contrary to the 1D case, we will not present the formal analysis allowing to guess that the limit equation is ZK. The principle is indeed identical, but the computations become quite lengthy. Let us refer to \cite{LLS} for that point.

We state directly the theorem asserting the convergence to ZK:

\begin{thm}
\label{ZK}
Let $(f_{\eps,0})_{\eps \in (0,1)}$ be a family of non-negative initial data satisfying the assumptions of Theorem \ref{Cauchy3D} and such that there exists $C>0$ with:
\begin{equation}
\mathcal{E}_\eps(0) \leq C, \, \, \forall \eps \in (0,1).
\end{equation}

We denote by $(f_\eps)_{\eps \in (0,1)}$ a sequence of global weak solutions to \eqref{VP3D} given by Theorem \ref{Cauchy3D}.
Let $\mathcal{H}_\eps$ be the relative entropy defined by the functional:
\begin{equation}
\begin{aligned}
\mathcal{H}_\eps (t) &:= \frac{1}{2} \int f_\eps \Big[| v_{1} - u^{(1)}_{1} - \eps u^{(2)}_{1} |^2 \\
&\quad \quad \quad \quad+ | v_2 - \sqrt{\eps} u^{(1)}_{2} - \eps u^{(2)}_{2} |^2  +  | v_3 - \sqrt{\eps} u^{(1)}_3 - \eps u^{(2)}_3|^2\Big] \, dv \, dx \\ 
&+ \frac{1}{2} \eps \int | \na_x \phi_\eps - \na_x \phi_1 - \eps \na_x \phi_2 -\eps^2 \na_x \phi_3 |^2 \,dx \\
&+ \frac{1}{\eps^2} \int  \left( e^{\eps \phi_\eps} \log (e^{\eps (\phi_\eps)}/e^{\eps(\phi_1 + \eps \phi_2+\eps^2 \phi_3)}) - e^{\eps \phi_\eps} + e^{\eps(\phi_1 + \eps \phi_2+\eps^2 \phi_3)}\right) \, dx,
\end{aligned}
\end{equation}
where 
\begin{multline*}
(\phi_1, u^{(1)}_1,u^{(1)}_2,u^{(1)}_3, u^{(2)}_1, u^{(2)}_2, u^{(2)}_3,   \phi_2, \phi_3) \\
\in [C([0,T_0[, H^{s+2}(\T^3))]^2 \times [C([0,T_0[, H^{s+1}(\T^3))]^2 \times[C([0,T_0[, H^s(\T^3))]^5
\end{multline*}
(with $s>3/2+1$, $T_0>0$) satisfy the following system on $[0,T_0[$:
\begin{equation}
\label{iden2}
\left\{
    \begin{aligned}
&  2\partial_t \phi_1 + 2 \phi_1 \pa_x \phi_1 + \pa_{x_1}(\Delta + \Delta_\perp)  u_1 = 0,  \\
&   u^{(1)}_1 = \phi_1, \quad u^{(1)}_2 = -\pa_{x_3} \phi_1, \quad  u^{(1)}_3 = \pa_{x_2} \phi_1, \\
& \phi_2  = \pa^2_{x_1 x_1} \phi_1, \\
&  u^{(2)}_2 = \pa^2_{x_2 x_2} \phi_1, \quad u^{(2)}_3 = \pa^2_{x_3 x_3} \phi_1, \\ 
&  \pa_{x_1}  u^{(2)}_1 = \pa_t u_1 + u_1 \pa_x u_1 + \pa^3_{x_1 x_1 x_1} \phi_1, \\
 &\pa_{x_1} \phi_3= - \pa_t \phi_2  - \pa_{x_1}(u^{(2)}_1  \phi_1) - u^{(1)}_1 \pa_{x_1} \phi_2 -  u^{(1)}_2 \pa_{x_2} \phi_2 + u^{(1)}_3 \pa_{x_3} \phi_2 .
   \end{aligned}
  \right.
\end{equation}
Then there exist $C_1,C_2>0$, such that for all $\eps \in (0,1)$, 
\begin{equation}
\forall t \in [0,T_0[, \, \, \mathcal{H}_\eps (t) \leq \int_0^t  (C_1 \mathcal{H}_\eps (s) \, ds  + C_2 \sqrt{\eps}) \, ds.
\end{equation}
Assuming in addition that there exists $C_3$ such that for any $\eps \in (0,1)$:
\begin{equation}
\mathcal{H}_\eps (0) \leq C_3  \sqrt{\eps},
\end{equation}
then we obtain for any $\eps \in (0,1)$:
\begin{equation}
\label{gron2}
\forall t \in [0,T_0[, \quad \mathcal{H}_\eps (t) \leq C_3   e^{C_1 t}  \sqrt{\eps} + C_2 \frac{e^{C_1 t}-1}{C_1} \sqrt{\eps}.
\end{equation}

\end{thm}

In this theorem, we need a large number of correctors in the relative entropy, precisely because of the nonlinear exponential term in the Poisson equation (compare with the case of Theorem \ref{KdV}, where this term is linearized). In order to get functions satisfying \eqref{iden2}, we can proceed as follows:
\begin{itemize}
\item The existence of $\phi_1$ is ensured by Theorem \ref{ZKcauchy} (actually the first equation of \eqref{iden2} is slightly different from \eqref{ZKeq}, but we can come down to \eqref{ZKeq} by using some standard change of variables, see for instance \cite{LS}).
\item We observe that the six correctors  $(u^{(1)}_2,u^{(1)}_3, u^{(2)}_1, u^{(2)}_2, u^{(2)}_3,   \phi_2)$ have their value which is uniquely determined (contrary to the $1D$ case). We accordingly set:
\begin{equation*}
\left\{
    \begin{aligned}
&   u^{(1)}_1 = \phi_1, \quad u^{(1)}_2 = -\pa_{x_3} \phi_1, \quad  u^{(1)}_3 = \pa_{x_2} \phi_1, \\
& \phi_2  = \pa^2_{x_1 x_1} \phi_1, \\
&  u^{(2)}_2 = \pa^2_{x_2 x_2} \phi_1, \quad u^{(2)}_3 = \pa^2_{x_3 x_3} \phi_1, \\ 
&  \pa_{x_1}  u^{(2)}_1 = \pa_t u_1 + u_1 \pa_x u_1 + \pa^3_{x_1 x_1 x_1} \phi_1,.\\
   \end{aligned}
  \right.
\end{equation*}

\item Finally, $\phi_3$ is a high order corrector whose value is imposed by the last equation of \eqref{iden2}.
\end{itemize}

\begin{rque} Note that a global well-posedness result for ZK in $\T^3$ would yield global in time estimates (that is $T_0=+\infty$) in this theorem, which would be significant for logarithmically growing times, as for Theorem \ref{KdV}.
\end{rque}

\begin{rque}
From \eqref{gron2}, we deduce that:
\begin{multline}
\frac{1}{\eps^2} \int  \left( e^{\eps \phi_\eps} \log (e^{\eps (\phi_\eps)}/e^{\eps(\phi_1 + \eps \phi_2+\eps^2 \phi_3)}) - e^{\eps \phi_\eps} + e^{\eps(\phi_1 + \eps \phi_2+\eps^2 \phi_3)}\right) \, dx \\
\leq C_3   e^{C_1 t}  \sqrt{\eps} + C_2 \frac{e^{C_1 t}-1}{C_1} \sqrt{\eps}.
\end{multline}
Following the terminology in the Boltzmann literature (see for instance the book \cite{SRbook}), this roughly means that $e^{\eps \phi_\eps}$ ``converges entropically'' to $e^{\eps(\phi_1 + \eps \phi_2+\eps^2 \phi_3)}$.

From the elementary inequality \eqref{ineq} (which will be given later), we can also deduce the control:
\begin{equation}
\frac{1}{\eps^2} \int  \left( e^{\frac{1}{2} \eps \phi_\eps} - e^{\frac{1}{2} \eps(\phi_1 + \eps \phi_2+\eps^2 \phi_3)}\right)^2 \, dx 
\leq C_3   e^{C_1 t}  \sqrt{\eps} + C_2 \frac{e^{C_1 t}-1}{C_1} \sqrt{\eps}.
\end{equation}

\end{rque}

We have as before the corollary:

\begin{coro}
\label{coro2}
Making the same assumptions as in the previous theorem, we obtain the  weak convergences:
\begin{equation}
\left\{
\begin{aligned}
&\rho_\eps \rightharpoonup_{\eps \rightarrow 0} 1  \text{  in  } L^\infty_t \mathcal{M}^1 \text{  weak-*}, \\
&J_\eps \rightharpoonup_{\eps \rightarrow 0} (\phi_1,0,0)  \text{  in  } L^\infty_t \mathcal{M}^1 \text{  weak-*}.
\end{aligned}
\right.
\end{equation}
\end{coro}

Up to some obvious modifications, the proof of Corollary \ref{coro2} is similar to that of Corollary \ref{coro1}, and therefore we omit it.

\subsection{Proof of Theorem \ref{ZK}}
Relying on the fact the energy $\mathcal{E}_\eps(t)$ is a non-increasing function of time (and that $\Hc_\eps$ is built as a modulation of the energy), we have for all $t\in [0,T_0[$ and all $\eps \in (0,1)$:
\begin{align*}
\Hc_\eps (t) &\leq \Hc_\eps(0) +  \int_0^t  \int \pa_t \left[ f_\eps \left( \frac{1}{2} |u^{(1)}_1 + \eps u^{(2)}_1 |^2 - v_1 \, (u^{(1)}_1 + \eps u^{(2)}_1) \right) \right] \, dv \, dx  \, ds \\
&+  \int_0^t \int \pa_t \left[ f_\eps \left( \frac{1}{2} |\sqrt{\eps} u^{(1)}_2 + \eps u^{(2)}_2 |^2 - v_2 \, (\sqrt{\eps} u^{(1)}_2 + \eps u^{(2)}_2) \right) \right] \, dv \, dx  \, ds \\
&+  \int_0^t \int \pa_t \left[ f_\eps \left( \frac{1}{2} |\sqrt{\eps} u^{(1)}_3 + \eps u^{(2)}_3 |^2 - v_3 \, (\sqrt{\eps} u^{(1)}_3 + \eps u^{(2)}_3) \right) \right] \, dv \, dx  \, ds \\
&+ \eps  \int_0^t \int \pa_t \left[ \frac{1}{2} | \na_x \phi_1 + \eps \na_x \phi_2 |^2 - \na_x \phi_\eps \cdot (\na_x \phi_1 + \eps \na_x \phi_2)\right ] \, dx  \, ds \\
&+ \frac{1}{\eps^2}  \int_0^t \int \pa_t \left [e^{\eps \phi_\eps} \log (1/ e^{\eps(\phi_1 + \eps \phi_2 + \eps^2 \phi_3)}) + e^{\eps(\phi_1 + \eps \phi_2 +\eps^2 \phi_3)}\right] \, dx  \, ds\\
&:=\Hc_\eps(0) + \int_0^t ( I_1^1 + I_1^2 + I_1^3 + I_2 + I_3) \, ds.
\end{align*}

The general strategy in the proof will be to keep (without making approximation) all dangerous modulated terms of the form
\begin{multline*}
\eps^{\alpha}\int f_\eps \begin{pmatrix} u^{(1)}_1 + \eps u^{(2)}_1 - v_1 \\ \sqrt{\eps} u^{(1)}_2 + \eps u^{(2)}_2 -v_2 \\  \sqrt{\eps} u^{(1)}_3 + \eps u^{(2)}_3 -v_3\end{pmatrix} [...] \, dv \, dx \, \\ \text  {and  } \\
\eps^\alpha \int (-e^{\eps \phi_\eps} + e^{\eps (\phi_1 + \eps \phi_2 + \eps^2 \phi_3)})[...] \, dx,
\end{multline*}
where $[...] $ contains some expression independent of $\eps$, as soon as $\alpha\leq 0$. Then \eqref{iden2} is precisely designed so that all terms exactly vanish in the end.

On the contrary, for $\alpha>0$, these terms will be of order $\eps^\alpha$ (and thus small). 
For the first type of terms, this can be seen with the same argument as in the proof of theorem \ref{KdV}, namely the uniform (in $\eps$) bounds on the $L^1$ norm of $\rho_\eps$ and $J_\eps$. For the second type of terms, one has to use the Poisson equation satisfied by $\phi_\eps$ and use the bound on the electric energy. Indeed, given some smooth function $\Psi$, we can write:
\begin{align*}
\eps^\a \int e^{\eps \phi_\eps} \Psi \, dx =  \eps^\a \int \eps^2 \Delta \phi_\eps \Psi \, dx  + \eps^\a \int \rho_\eps \Psi \, dx.
\end{align*}
The second term is clearly of order $\eps^\a$, using the uniform $L^1$ bound on $\rho_\eps$. On the other hand, by integration by parts, we have for the first term
\begin{align*}
 \eps^\a \int \eps^2 \Delta \phi_\eps \Psi \, dx = \eps^{\a+2} \int \na \phi_\eps \cdot \na \Psi \, dx, 
\end{align*}
which is of order $\eps^ {\a+1}$ using the uniform bound obtained thanks to the conservation of energy:
$$
\eps \int |\na \phi_\eps |^2 \, dx \leq \mathcal{E}_\eps(0) \leq C.
$$
Finally any term like $\eps^\alpha \int e^{\eps (\phi_1 + \eps \phi_2 + \eps^2 \phi_3)}[...] \, dx$ is clearly of order $\eps^\alpha$.

We start by studying separately $I_1^1 + I_1^2 + I_1^3$, $I_2$ and $I_3$. All computations are justified using only the local conservation laws \eqref{LC1'} and \eqref{LC2'}.

\medskip
{\bf Study of $I_1^1 + I_1^2 + I_1^3$}.

With similar computations as in the proof of Theorem \ref{KdV}, we obtain the identity
\begin{align*}
&I_1^1 + I_1^2 + I_1^3 = \\
&\frac{1}{\eps} \int f_\eps \begin{pmatrix} u^{(1)}_1 + \eps u^{(2)}_1 - v_1 \\ \sqrt{\eps} u^{(1)}_2 + \eps u^{(2)}_2 -v_2 \\  \sqrt{\eps} u^{(1)}_3 + \eps u^{(2)}_3 -v_3\end{pmatrix} \cdot \begin{pmatrix} - \pa_{x_1} u^{(1)}_1 + \pa_{x_1} \phi_1 \\ -u^{(1)}_3 + \pa_{x_2} \phi_1 \\ u^{(1)}_2 + \pa_{x_3} \phi_1 \end{pmatrix} \, dv \, dx \\
&- \int \frac{1}{\eps} \rho_\eps u^{(1)} \cdot \nabla_x \phi_1 \, dx - \int  \rho_\eps u^{(2)} \cdot \nabla_x \phi_1 \, dx+ \int \frac{1}{\eps} J_\eps \cdot \nabla_x \phi_1 \, dx \\
&+ \frac{1}{\sqrt{\eps}} \int f_\eps \begin{pmatrix} \sqrt{\eps} u^{(1)}_2 + \eps u^{(2)}_2 -v_2 \\  \sqrt{\eps} u^{(1)}_3 + \eps u^{(2)}_3 -v_3\end{pmatrix} \cdot \begin{pmatrix} -u^{(2)}_3 - \pa_{x_1} u^{(1)}_2 \\ u^{(2)}_2 - \pa_{x_1} u^{(1)}_3  \end{pmatrix} \, dv \, dx \\
&+ \int f_\eps \begin{pmatrix} u^{(1)}_1 + \eps u^{(2)}_1 - v_1 \\ \sqrt{\eps} u^{(1)}_2 + \eps u^{(2)}_2 -v_2 \\  \sqrt{\eps} u^{(1)}_3 + \eps u^{(2)}_3 -v_3\end{pmatrix} \cdot \begin{pmatrix} \pa_t u^{(1)}_1 + u^{(1)}_1 \pa_{x_1} u^{(1)}_1 - \pa_{x_1} u^{(1)}_2 + \pa_{x_1} \phi_2 \\ - \pa_{x_1} u^{(2)}_2 + \pa_{x_2} \phi_2 \\  - \pa_{x_1} u^{(2)}_3 + \pa_{x_3} \phi_2 \end{pmatrix} \, dv \, dx  \\
&+ \int J_\eps \cdot \na_x \phi_2 \, dx  -  \int \rho_\eps u^{(1)} \cdot \na_x \phi_2 \, dx  - \eps \int \rho_\eps u^{(2)} \cdot \na_x \phi_2 \, dx \\
& - \int f_\eps ( u^{(1)}_1 + \eps u^{(2)}_1 - v_1)^2  \pa_x u^{(1)}_1 \, dx\\
&- \int \frac{1}{\eps} \rho_\eps E_\eps \cdot u^{(1)}  \, dx  - \int  \rho_\eps E_\eps  \cdot u^{(2)} \, dx \\
& + \mathcal{O}(\eps).
\end{align*}
In this equation, $\mathcal{O}(\eps)$ is as usual a notation for all terms that can bounded by $C \eps$, where $C$ is a positive constant. We also denote here: 
$$u^{(1)} := (u^{(1)}_1, \sqrt{\eps} u^{(1)}_2, \sqrt{\eps} u^{(1)}_2)  \text{  and  } u^{(2)} := (u^{(2)}_1,u^{(2)}_2,  u^{(2)}_2).$$

We observe that we can write:
\begin{multline*}
- \int \frac{1}{\eps} \rho_\eps E_\eps \cdot u^{(1)}  \, dx  - \int \frac{1}{\eps}  \rho_\eps  u^{(1)} \cdot \na_x \phi_1 \, dx -  \int \rho_\eps u^{(1)} \cdot\na_x \phi_2  \, dx \\
= - \frac{1}{\eps} \int \rho_\eps u_1 \cdot ( \na_x \phi_1 + \eps \na_x \phi_2  - \na_x \phi_\eps).
\end{multline*}

\medskip
{\bf Study of $I_2$}.

We obtain:
\begin{multline*}
I_2= \eps \int \pa_t (\na_x \phi_1 + \eps \na_x  \phi_2 + \eps^2 \na_x \phi_3)\cdot  (-\na_x  \phi_\eps  + \na_x  \phi_1 + \eps \na_x \phi_2 + \eps^2 \na_x \phi_3) \, dx \\
- \eps \int \pa_t \na_x  \phi_\eps \cdot (  \na_x  \phi_1 + \eps \na_x  \phi_2 + \eps^2 \na_x \phi_3) \, dx.
\end{multline*}

We have the easy bound:
\begin{align*}
&\eps \left|\int \pa_t (\na_x \phi_1 + \eps \na_x  \phi_2 + \eps^2 \na_x \phi_3)\cdot  (-\na_x  \phi_\eps  + \na_x  \phi_1 + \eps \na_x \phi_2 + \eps^2 \na_x \phi_3) \, dx\right| \\
&\leq C \left[\eps  \int (\pa_t  (\na_x \phi_1 + \eps \na_x  \phi_2 + \eps^2 \na_x \phi_3))^2 \, dx + \eps \int (-\pa_x  \phi_\eps + \pa_x  \phi_1 + \eps \pa_x \phi_2)^2 \, dx \right] \\
&\leq \mathcal{O}(\eps) + C \Hc_\eps(t).
\end{align*}

\medskip
{\bf Study of $I_3$}.

We can compute:
\begin{align*}
I_3 &= \frac{1}{\eps} \int (-e^{\eps \phi_\eps} + e^{\eps (\phi_1 + \eps \phi_2 + \eps^2 \phi_3)}) \pa_t (\phi_1 + \eps \phi_2 +\eps^2 \phi_3) \, dx \\
& - \frac{1}{\eps} \int \pa_t (e^{\eps \phi_\eps})( \phi_1 + \eps \phi_2 +\eps^2 \phi_3) \, dx \\
&= I_3^1 + I_3^2.
\end{align*}

The first term can be recast as follows:
\begin{align*}
I_3^1 =  \frac{1}{\eps} \int (-e^{\eps \phi_\eps} + e^{\eps (\phi_1 + \eps \phi_2 + \eps^2 \phi_3)}) \pa_t \phi_1  \, dx  + \int (-e^{\eps \phi_\eps} + e^{\eps (\phi_1 + \eps \phi_2 + \eps^2 \phi_3)}) \pa_t   \phi_2 \, dx \\
+   \eps \int (-e^{\eps \phi_\eps} + e^{\eps (\phi_1 + \eps \phi_2 + \eps^2 \phi_3)}) \pa_t   \phi_3. 
\end{align*}

For the second one, we use the Poisson equation (with a time derivative) and the conservation of charge in shifted variables, that we recall below:
\begin{equation*}
\left\{
\begin{aligned}
&\pa_t (e^{\eps \phi_\eps})=  \eps^2 \Delta \pa_t \phi_\eps + \pa_t \rho_\eps, \\
 &\frac{1}{\eps} \pa_t \rho_\eps =  \frac{1}{\eps^2}\pa_{x_1} \rho_\eps -  \frac{1}{\eps} \na_x \cdot J_\eps,
\end{aligned}
\right.
\end{equation*}
to obtain:
\begin{multline*}
I_3^2 = -\frac{1}{\eps^2} \int \pa_{x_1} \rho_\eps ( \phi_1 + \eps \phi_2 + \eps^2 \phi_3 ) \, dx   \\
+ \frac{1}{\eps} \int \na_x \cdot J_\eps  \phi_1 \, dx + \int \na_x \cdot J_\eps \phi_2 \, dx  + \eps \int \na_x \cdot J_\eps \,  \phi_3 \, dx \\
-\eps \int \pa_t \Delta \phi_\eps (\phi_1 + \eps \phi_2 + \eps^2 \phi_3).
\end{multline*}

As in the proof of Theorem \ref{KdV}, every term in $I_3^2$ but the first one gets simplified with some other ones obtained in $I_1$ and $I_2$.

\medskip
{\bf Study of the remaining significant terms}.

There remain to study the following potentially harmful terms:
\begin{align*}
K_1 &:= - \int \rho_\eps u^{(2)}\cdot  \na_x( \phi_1  +\eps \phi_2 ) \, dx, \quad
K_2 := - \int \rho_\eps E_\eps \cdot u^{(2)}\, dx, \\
K_3 &:= - \frac{1}{\eps} \int  \rho_\eps u^{(1)} \cdot  (  \na_x \phi_1 + \eps \na_x \phi_2 - \na_x \phi_\eps) \, dx, \\
K_4 &:= - \frac{1}{\eps^2} \int \pa_{x_1} \rho_\eps (\phi_1 + \eps \phi_2  + \eps^2 \phi_3) \, dx, \\
K_5 &:= \frac{1}{\eps} \int (-e^{\eps \phi_\eps} + e^{\eps (\phi_1 + \eps \phi_2 + \eps^2 \phi_3)}) \pa_t \phi_1  \, dx. \\
K_6 &:=  \int (-e^{\eps \phi_\eps} + e^{\eps (\phi_1 + \eps \phi_2 + \eps^2 \phi_3)}) \pa_t   \phi_2 \, dx.
\end{align*}

We start with $K_1 + K_2$, using the Poisson equation:
\begin{align*}
K_1 + K_2  &= - \int e^{\eps \phi_\eps} u^{(2)}\cdot \na_x ( \phi_1  + \eps \phi_2 - \phi_\eps) \, dx  \\
&+ \eps^2 \int \Delta \phi_\eps u^{(2)} \cdot ( \na_x \phi_1 + \eps \na_x \phi_2 - \na_x \phi_\eps) \, dx. \\
&= L_1 + L_2.
\end{align*}

We have:
\begin{align*}
L_1  =  - \int e^{\eps \phi_\eps} u^{(2)}\cdot \na_x  \phi_1  \, dx - \eps  \int e^{\eps \phi_\eps} u^{(2)}\cdot \na_x  \phi_2   \, dx \\
- \frac{1}{\eps} \int  e^{\eps \phi_\eps} \na_x \cdot u^{(2)} \, dx.
\end{align*}

For $L_2$, we have
\begin{align*}
L_2  &=   \eps^2 \int \Delta \phi_\eps u^{(2)} \cdot ( \na_x \phi_1 + \eps \na_x \phi_2 ) \, dx
-  \eps^2 \int \Delta \phi_\eps u^{(2)} \cdot   \na_x \phi_\eps \, dx \\
&:= L_2^1 + L_2^2.
\end{align*}

By integration by parts, we get:
\begin{align*}
|L_2^1|  &=  \left| \eps^2 \int \na \phi_\eps \cdot u^{(2)}  ( \Delta \phi_1 + \eps \Delta \phi_2 ) \, dx \right| \\
&\leq C \left( \eps^2 \int | \na \phi_\eps |^2 \, dx + C \eps  \right) \\
&\leq  C (\eps \mathcal{E}_\eps(t)  + C \eps) \leq C \eps.
\end{align*}
Note here that we have used the Lipschitz bound on the second derivative of $\phi_1$ and $\phi_2$.

For $L_2^2$, we rely on the usual trick to write:
\begin{align*}
L_2^2  &=  -\frac{1}{2} \eps^2 \int \na_x |\na_x \phi_\eps|^2  \cdot u^{(2)} \, dx \\
&= \frac{1}{2} \eps^2 \int  |\na_x \phi_\eps|^2  \, \na_x \cdot u^{(2)} \, dx,
\end{align*}
from which we deduce that
$$
|L_2^2| \leq C \eps.
$$

Using the Poisson equation for $K_3$:


\begin{align*}
K_3 &=  - \frac{1}{\eps} \int e^{\eps \phi_\eps}   u^{(1)}  \cdot  ( \na_x \phi_1 + \eps \na_x \phi_2  - \na_x \phi_\eps) \, dx  \\
&+  \int  \eps \Delta \phi_\eps u^{(1)}  \cdot  (\na_x \phi_1 + \eps \na_x \phi_2 - \na_x \phi_\eps) \, dx \\
&:= K_3^1 + K_3^2.
\end{align*}

Concerning the first term, we write the decomposition:
\begin{multline*}
K_3^1 =  - \frac{1}{\eps} \int e^{\eps \phi_\eps}   u^{(1)}  \cdot  \na_x \phi_1   \, dx  \\
-  \int e^{\eps \phi_\eps}   u^{(1)}  \cdot   \na_x  \phi_2  \, dx  + \frac{1}{\eps} \int e^{\eps \phi_\eps}   u^{(1)}  \cdot  \na_x \phi_\eps \, dx \\
:= J_1 + J_2 + J_3.
\end{multline*}

We start by the study of $J_1$:
\begin{align*}
J_1&= \frac{1}{\eps} \int \left( e^{\eps (\phi_1 + \eps \phi_2 + \eps^2 \phi_3)}- e^{\eps \phi_\eps}\right)  u^{(1)}_1   \pa_{x_1} \phi_1    \, dx  \\
&- \frac{1}{\eps} \int e^{\eps(\phi_1 - \eps \phi_2)}  u^{(1)}_1   \pa_{x_1} \phi_1   \, dx  - \frac{1}{\eps} \int  e^{\eps \phi_\eps}(u^{(1)}- u^{(1)}_1 e_1)   \cdot \na_{x}  \phi_1   \, dx.
\end{align*}



Clearly, we have (for instance using a Taylor inequality): 
\begin{align*}
- \frac{1}{\eps} \int e^{\eps(\phi_1 - \eps \phi_2)}  u^{(1)}_1   \pa_{x_1} \phi_1    \, dx =  -\frac{1}{\eps} \int u^{(1)}_1   \pa_{x_1} \phi_1    \, dx +  \int \phi_1 u^{(1)}_1   \pa_{x_1} \phi_1    \, dx +   \Oc(\eps).
\end{align*}
In the end, we will take $u^{(1)}_1   = \phi_1$, so the terms of this latest expression which are of order $\Oc(1/\eps)$ and $\Oc(1)$ are exactly equal to $0$.

For $J_2$, it is sufficient to write:
\begin{align*}
J_2&= -  \int e^{\eps \phi_\eps}   u^{(1)}_1    \pa_{x_1} \phi_2  \, dx  -  \int e^{\eps \phi_\eps}   (u^{(1)}-u^{(1)}_1 e_1)  \cdot   \na_x \phi_2  \, dx .
\end{align*}





We decompose the last term in the following way
\begin{align*}
J_3&= \frac{1}{\eps} \int e^{\eps \phi_\eps}  \phi_1 \pa_{x_1} \phi_\eps  \\
&+  \frac{1}{\eps} \int e^{\eps \phi_\eps}  (u^{(1)}_1-\phi_1) \pa_{x_1} \phi_\eps   \\
&+ \frac{1}{\eps} \int e^{\eps \phi_\eps} (u^{(1)} - u^{(1)}_1 e_1) \cdot \na_{x} \phi_\eps .
\end{align*}

The first term could be dangerous, but will disappear using a term coming from a term of $K_4$ (first term of $K_4^1$ below). The second one will be equal to $0$ since we take $u^{(1)}_1=\phi_1$.

Gathering the pieces together, we finally obtain:
\begin{multline*}
 \frac{1}{\eps} \int e^{\eps \phi_\eps} (u^{(1)} - u^{(1)}_1 e_1) \cdot \na_{x} ( \phi_\eps -\phi_1-\eps \phi_2) = \\
- \frac{1}{ \eps^{3/2}} \int e^{\eps \phi_\eps}( \pa_{x_2} u^{(1)} _2 + \pa_{x_3} u^{(1)} _3)  \, dx \\
- \frac{1}{\sqrt \eps} \int e^{\eps \phi_\eps}\left( u^{(1)} _2 \pa_{x_2} \phi_1 + u^{(1)} _3 \pa_{x_3} \phi_1\right)  \, dx \\
-  {\sqrt \eps} \int e^{\eps \phi_\eps}\left( u^{(1)} _2 \pa_{x_2} \phi_2 + u^{(1)} _3 \pa_{x_3} \phi_2 \right)  \, dx.
\end{multline*}




Let us now turn to the treatment of $K_3^2$:
\begin{align*}
K_3^2 =  \int \eps \Delta \phi_\eps u^{(1)} \cdot ( \na_x \phi_1 + \eps \na_x \phi_2) \, dx -  \eps/2 \int \na_x |\na_x \phi_\eps|^2 \cdot  u^{(1)} \, dx.
\end{align*}
This term is treated exactly as $L_2$, but we have to be careful since it is singular in $\eps$. Here, rather than bounding by the energy, we shall rely on a bound by the modulated energy. We have

\begin{multline*}
\eps/2 \int \na_x |\na_x \phi_\eps|^2 \cdot  u^{(1)} \, dx \\
= - \eps/2 \int  |\na_x \phi_\eps|^2  \na_x \cdot  u^{(1)} \, dx \\
- \eps/2 \int  |\na_x \phi_\eps  - \na_x \phi_1 - \eps \na_x \phi_2 -\eps^2 \na_x \phi_3|^2  \na_x \cdot  u^{(1)} \, dx \\
+  \eps/2 \int  |  \na_x \phi_1 + \eps \na_x \phi_2 + \eps^2 \na_x \phi_3|^2  \na_x \cdot  u^{(1)} \, dx \\
- \eps  \int  \na_x \phi_\eps \cdot ( \na_x \phi_1 +\eps \na_x \phi_2+ \eps^2 \na_x \phi_3)  \na_x \cdot  u^{(1)} \, dx.
\end{multline*}


We shall bound the very last term using the Cauchy-Schwarz inequality and the various Lipschitz bounds:
\begin{align*}
\left|\eps  \int  \na_x \phi_\eps \cdot ( \na_x \phi_1 +\eps \na_x \phi_2+ \eps^2 \na_x \phi_3)  \na_x \cdot  u^{(1)} \, dx\right| \\
\leq C \sqrt{\eps} \sqrt{\eps \int |\na_x \phi_\eps|^2 \,dx} \leq C \sqrt{\eps} \mathcal{E}_\eps(t) \leq C \sqrt{\eps} .
\end{align*}
Therefore we obtain:
\begin{equation*}
|K_3^2 | \leq C \Hc_\eps(t) + C \sqrt{\eps}.
\end{equation*}

Finally, concerning $K_4$ we get:
\begin{align*}
K_4 &= \frac{1}{\eps^2} \int \rho_\eps ( \pa_{x_1} \phi_1 + \eps \pa_{x_1} \phi_2 + \eps^2 \pa_{x_1} \phi_3 ) \, dx \\
&= \frac{1}{\eps^2} \int e^{\eps \phi_\eps} ( \pa_{x_1} \phi_1 + \eps \pa_{x_1} \phi_2+  \eps^2 \pa_{x_1} \phi_3) \, dx \\
& -   \int \Delta \phi_\eps ( \pa_{x_1} \phi_1 + \eps \pa_{x_1} \phi_2+  \eps^2 \pa_{x_1} \phi_3) \, dx  \\
&:= K_4^1 + K_4^2.
\end{align*}

The first term can be decomposed as:
\begin{align*}
K_4^1= - \frac{1}{\eps} \int \pa_{x_1} \phi_\eps  e^{\eps \phi_\eps}   \phi_1  \, dx \\
+\frac{1}{\eps} \int e^{\eps \phi_\eps}  \pa_{x_1} \phi_2  \, dx\\
+ \int e^{\eps \phi_\eps}  \pa_{x_1} \phi_3 \, dx.
\end{align*}

The second one can be recast as:
\begin{align*}
K_4^2=- \int  \phi_\eps \Delta \pa_{x_1} \phi_1  \, dx + \Oc(\sqrt{\eps}).
\end{align*}

We shall focus our attention on the important term:
\begin{align*}
L &:= -\int  \phi_\eps \Delta \pa_{x_1} \phi_1 \, dx\\
&= -\frac{1}{\eps} \int e^{\eps \phi_\eps}   \Delta \pa_{x_1} \phi_1  \, dx+  \frac{1}{\eps} \int (e^{\eps \phi_\eps} - \eps \phi_\eps)  \Delta \pa_{x_1} \phi_1 \, dx.
\end{align*}


We have the following technical result, allowing to consider the second term above as an error term:

\begin{lem}
\label{tech}
There exists $C>0$ such that for any $\eps \in (0,1)$:
\begin{align*}
  &\frac{1}{\eps} \left|\int (e^{\eps \phi_\eps} - \eps \phi_\eps)  \Delta \pa_{x_1} \phi_1 \, dx\right| \leq C \sqrt{\eps}.
\end{align*}


\end{lem}

\begin{proof}[Proof of Lemma \ref{tech}]

The naive idea would be to use the Taylor expansion
$$
e^x \sim_{0} 1+ x + \frac{1}{2} x^2.
$$
but we cannot say that $\eps \| \phi_\eps \|_\infty \ll 1$. (Even worse, we do not have any $L^2$ control on $\phi_\eps$.)
Instead, we will only rely on the bounds given by the conservation of energy.  The classical inequality (valid for $x,y>0$) will be very useful:
\begin{equation}
\label{ineq}
(\sqrt{x}-\sqrt{y}) \leq x \log (x/y) - x + y.
\end{equation}

We shall  write the decomposition

\begin{multline*}
\frac{1}{\eps} \int (e^{\eps \phi_\eps} - \eps \phi_\eps)  \Delta \pa_{x_1} \phi_1 \, dx = \\
\frac{1}{\eps} \int (e^{\eps \phi_\eps} -2 e^{\frac{1}{2}\eps \phi_\eps})  \Delta \pa_{x_1} \phi_1 \, dx \\
+\frac{1}{\eps}  \int (e^{\frac{1}{2}\eps \phi_\eps} - \frac{1}{2}\eps \phi_\eps) 2 \Delta \pa_{x_1} \phi_1 \, dx.
\end{multline*}

We first recast the fist term as follows:
\begin{align*}
\frac{1}{\eps}  \int (e^{\eps \phi_\eps} -2 e^{\frac{1}{2}\eps \phi_\eps})  \Delta \pa_{x_1} \phi_1 \, dx &= \int (e^{\eps \phi_\eps} -2 e^{\frac{1}{2}\eps \phi_\eps}+1)  \Delta \pa_{x_1} \phi_1 \, dx \\
&= \frac{1}{\eps} \int ( e^{\frac{1}{2}\eps \phi_\eps}-1)^2  \Delta \pa_{x_1} \phi_1 \, dx .
\end{align*}

Using \eqref{ineq} and the Lipschitz bound on the second order derivative of $\phi_1$, we obtain:
\begin{align*}
 \frac{1}{\eps} \int ( e^{\frac{1}{2}\eps \phi_\eps}-1)^2  \Delta \pa_{x_1} \phi_1 \, dx &\leq C \eps \int \frac{1}{\eps^2} (e^{\eps \phi_\eps}(\eps \phi_\eps-1) -1) \, dx\\
& \leq \eps \mathcal{E}_\eps(t) \\
& \leq  \eps \mathcal{E}_\eps(0) \\
& \leq C \eps.
\end{align*}

For the second term, we have by integration by parts:
\begin{align*}
&\frac{1}{\eps}  \int (e^{\frac{1}{2}\eps \phi_\eps} - \frac{1}{2}\eps \phi_\eps) 2 \Delta \pa_{x_1} \phi_1 \, dx = -  \int \na_x \phi_\eps (e^{\frac{1}{2}\eps \phi_\eps} - 1 ) \cdot \na_x \pa_{x_1} \phi_1 \, dx.
\end{align*}

Then we write, using $|ab| \leq \frac{1}{2} (\lambda a^2 + \frac{1}{\lambda} b^2)$, with $\a$ a parameter to be fixed:

\begin{multline*}
\left| \int \na_x \phi_\eps (e^{\frac{1}{2}\eps \phi_\eps} - 1 ) \cdot \na \pa_{x_1} \phi_1 \, dx \right| \\
\leq C \eps^{\a-1} \eps \int |\na_x \phi_\eps |^2  + C \eps^{-\a+2} \int \frac{1}{\eps^2}( e^{\eps \phi_\eps}(\eps \phi_\eps-1) +1) \, dx.
\end{multline*}
To optimize, it is clear that we should take $\a=3/2$. Thus, relying on the uniform bound given by the energy, we get:
\begin{align*}
&\left|\frac{1}{\eps}  \int \left(e^{\frac{1}{2}\eps \phi_\eps} - \frac{1}{2}\phi_\eps \right) 2 \Delta \pa_{x_1} \phi_1 \, dx \right| \leq C \sqrt{\eps}.
\end{align*}

The proof is therefore complete.
\end{proof}



\medskip
{\bf Conclusion.}

Finally, we impose all cancellations for the terms of the terms
$$\eps^{\alpha}\int f_\eps \begin{pmatrix} u^{(1)}_1 + \eps u^{(2)}_1 - v_1 \\ \sqrt{\eps} u^{(1)}_2 + \eps u^{(2)}_2 -v_2 \\  \sqrt{\eps} u^{(1)}_3 + \eps u^{(2)}_3 -v_3\end{pmatrix} [...] \, dv dx$$
for $\alpha = -1, -1/2, 0$ and 
$$
\eps^\beta \int (-e^{\eps \phi_\eps} + e^{\eps (\phi_1 + \eps \phi_2 + \eps^2 \phi_3)})[...] \, dx
$$
for $\beta =-3/2, -1, -1/2,0$,  which precisely means that we impose:
\begin{equation}
\left\{
\begin{aligned}
&- \pa_{x_1} u^{(1)}_1 + \pa_{x_1} \phi_1=0, \\ 
&-u^{(1)}_3 + \pa_{x_2} \phi_1 =0, \\ 
&u^{(1)}_2 + \pa_{x_3} \phi_1 =0, \\
&-u^{(2)}_3 - \pa_{x_1} u^{(1)}_2 =0,\\ 
&u^{(2)}_2 - \pa_{x_1} u^{(1)}_3 =0, \\
 &\pa_t u^{(1)}_1 + u^{(1)}_1 \pa_{x_1} u^{(1)}_1 - \pa_{x_1} u^{(1)}_2 + \pa_{x_1} \phi_2  =0, \\ 
 &- \pa_{x_1} u^{(2)}_2 + \pa_{x_2} \phi_2  =0, \\
  & - \pa_{x_1} u^{(2)}_3 + \pa_{x_3} \phi_2  =0, \\
  & \pa_{x_2} u^{(1)} _2 + \pa_{x_3} u^{(1)} _3 = 0, \\
  & \pa_t \phi_1 + \na_x \cdot u^{(2)} + \phi_1 \pa_{x_1} \phi_1 + \Delta \pa_{x_1} \phi_1 - \pa_{x_1} \phi_2  =0,\\
  &  u^{(1)} _2 \pa_{x_2} \phi_1 + u^{(1)} _3 \pa_{x_3} \phi_1 =0, \\
  &  \pa_t \phi_2  +   \phi_1 \pa_{x_1} \phi_2 + u^{(2)}_1 \pa_{x_1} \phi_1 + u^{(2)}_2 \pa_{x_2} \phi_1 + u^{(2)}_3 \pa_{x_3} \phi_1 - \pa_{x_1} \phi_3=0.\\
\end{aligned}
\right.
\end{equation}
All these indentities are exactly obtained as a consequence of \eqref{iden2}.

We end up with:
\begin{equation}
 \Hc_\eps(t) \leq  \Hc_\eps(0) + \int_0^t  (C_1 \Hc_\eps (t) + C_2 \sqrt\eps) \, ds,
\end{equation}
and arguing as in the proof of Theorem \ref{KdV}, this is over.

%


\section{Appendix A: From the Vlasov-Poisson equation to the Kadomstev-Petviashvili II equation}
\label{sec4}

Assuming a slow variation in the $x_2$ direction, one may end up with the following anisotropic long wave scaling for the Vlasov-Poisson system for ions.
For simplicity, we restrict here to the linearized Maxwell-Boltzmann law, but the same study can be performed for the full equations. The variables are $t\geq 0 , x \in \T^2, v \in \R^2$:
\begin{equation}
\label{VP2D}
\left\{
    \begin{aligned}
&  \eps \,  \partial_t f_\eps  - \pa_{x_1} f_\eps + \eps \, v_1 \pa_{x_1} f_\eps + \eps^{3/2} v_2 \pa_{x_2} f_\eps + E_\eps \cdot  \na_v f_\eps =0,  \\
&  E =(- \pa_{x_1} \phi_\eps, - \sqrt{\eps} \pa_{x_2} \phi_\eps), \\
&  -\eps^2 \pa^2_{x_1 x_1} \phi_\eps  -\eps^3 \pa^2_{x_2 x_2} \phi_\eps +  \eps \phi_\eps = \int_{\R^2}  f_\eps \, dv -1,\\
&    f_{\vert t=0} =f_0.\\
    \end{aligned}
  \right.
\end{equation}

The scaled energy of this system is the following functional:
\begin{equation}
\begin{aligned}
\mathcal{E}_\eps (t) &:= \frac{1}{2} \int f_\eps |v|^2 \, dv \, dx \\
&+ \frac{1}{2} \eps \int | \pa_{x_1} \phi_\eps  |^2 \,dx 
+ \frac{1}{2} \eps^2 \int | \pa_{x_2} \phi_\eps  |^2 \,dx + \frac{1}{2} \int  \phi_\eps^2 \, dx.
\end{aligned}
\end{equation}

We have the existence of global weak solutions, sharing the same properties of those given in Theorem \ref{Cauchy1D}.

For the KP-II equation, that is
\begin{equation}
 \pa_{x_1} \left( \partial_t \phi_1 +  \phi_1 \pa_{x_1} \phi_1+ \pa^3_{x_1 x_1 x_1 } \phi_1\right) + \pa^2_{x_2 x_2} \phi_1 = 0,
\end{equation}
our reference is an article by Bourgain \cite{Bour}, in which is proved the following theorem:
\begin{thm}
The ZK equation is globally well-posed in $H^s(\T^2)$ for $s\geq 0$.
\end{thm}
Once again, we will only use this theorem for large values of $s$.

As for the other cases, we obtain the rigorous convergence to KP-II, which is summarized in the following theorem:
\begin{thm}
\label{KP2}
Let $(f_{\eps,0})_{\eps \in (0,1)}$ be a family of non-negative initial data such that
\begin{equation}
\| f_{\eps,0} \|_{L^1 \cap L^\infty}<+ \infty, \quad \int f_{\eps,0} \, dv \,dx =1,
\end{equation}
and such that there exists $C>0$ with:
\begin{equation}
  \mathcal{E}_\eps(0) \leq C, \, \, \forall \eps \in (0,1).
\end{equation}

We denote by $(f_\eps)_{\eps \in (0,1)}$ a sequence of global weak solutions to \eqref{VP2D}.
Let $\mathcal{H}_\eps$ be the relative entropy defined by the functional:
\begin{equation}
\begin{aligned}
\mathcal{H}_\eps (t) &:= \frac{1}{2} \int f_\eps \left[| v_1 - u^{(1)}_1 - \eps u^{(2)}_1 |^2 + | v_2 - \sqrt{\eps} u^{(1)}_2 - \eps^{3/2} u^{(2)}_2 |^2 \right] \, dv \, dx \\
&+ \frac{1}{2} \eps \int | \pa_{x_1} \phi_\eps - \pa_{x_1} \phi_1 - \eps \pa_{x_1} \phi_2 |^2 \,dx 
+ \frac{1}{2} \eps^2 \int | \pa_{x_2} \phi_\eps - \pa_{x_2} \phi_1 - \eps \pa_{x_2} \phi_2 |^2 \,dx \\
&+ \frac{1}{2} \int  (\phi_\eps - \phi_1 - \eps \phi_2)^2 \, dx,
\end{aligned}
\end{equation}
where $(u_1,u_2, \phi_1, \phi_2) \in [C([0,+\infty[, H^{s+2}(\T^2))]^4$ (with $s>2$) satisfy the following system:
\begin{equation}
\left\{
    \begin{aligned}
& \pa_{x_1} \left(2 \partial_t \phi_1 + 3 \phi_1 \pa_{x_1} \phi_1+ \pa^3_{x_1 x_1 x_1 } \phi_1\right) + \pa^2_{x_2 x_2} \phi_1 = 0,  \\
&  u^{(1)}_1= \phi_1, \\
& \pa_{x_1} u^{(1)}_2 = \pa_{x_2} \phi_1, \\
&  \pa_{x_1} (u^{(2)}_1 - \phi_2) = \pa_t \phi_1 + \phi_1 \pa_{x_1} \phi_1. 
    \end{aligned}
  \right.
\end{equation}
Then there exist $C_1,C_2>0$, such that for all $\eps \in (0,1)$, 
\begin{equation}
\mathcal{H}_\eps (t) \leq \mathcal{H}_\eps (0)  + \int_0^t  (C_1 \mathcal{H}_\eps (s) \, ds  + C_2 \sqrt{\eps}) \, ds.
\end{equation}

Assuming in addition that there exists $C_3>0$ such that for any $\eps \in (0,1)$,
\begin{equation}
\mathcal{H}_\eps (0) \leq C_3 \sqrt{\eps}.
\end{equation}
Then we obtain for all $\eps \in (0,1)$:
\begin{equation}
\forall t \in [0,+ \infty[, \quad \mathcal{H}_\eps (t) \leq C_3   e^{C_1 t}  \sqrt{\eps} + C_2 \frac{e^{C_1 t}-1}{C_1} \sqrt{\eps},
\end{equation}
as well as the weak convergences: 
\begin{equation}
\left\{
\begin{aligned}
&\rho_\eps \rightharpoonup_{\eps \rightarrow 0} 1  \text{  in  } L^\infty_t \mathcal{M}^1 \text{  weak-*}, \\
&J_\eps \rightharpoonup_{\eps \rightarrow 0} (\phi_1,0)  \text{  in  } L^\infty_t \mathcal{M}^1 \text{  weak-*}.
\end{aligned}
\right.
\end{equation}

\end{thm}

The proof of Theorem \ref{KP2} follows from computations in the same spirit as the previous ones, and therefore we leave it to the reader.

\section{Appendix B: A KdV limit in the whole space $\R$}
\label{sec5}

All results stated in this paper are restricted to PDEs set in the torus for the space variable. The reason is that in all cases, in the end, the first moment $\rho_\eps$ (charge density) has to weakly converge to the constant $1$. This function is obviously not integrable in the whole space, and  the assumptions needed for our results to hold are actually not consistent in the whole space case.

It is nevertheless possible to slightly adapt the KdV limit of Section \ref{sec2} to handle thatcase. Keeping the same notations, we shall rely on the fact that the electric potential can be defined up to a constant, and use the following version of the energy
\begin{equation}
\begin{aligned}
\mathcal{F}_\eps (t) &:= \frac{1}{2} \int_{\R\times\R} f_\eps | v |^2   \, dv \, dx + \frac{1}{2} \eps \int_\R | \na_x \phi_\eps|^2 \,dx + \frac{1}{2}  \int_{\R\times\R} \left( \phi_\eps-\frac{1}{\eps} \right)^2 \,dx \, dx,
\end{aligned}
\end{equation}
This means that $\phi_\eps \in \frac{1}{\eps} + L^2(\R)$ (instead of  the more usual $\phi_\eps \in L^2(\R)$).

The Ansatz in that case for the formal computations now corresponds to the following one:
\begin{equation}
\left\{
\begin{aligned}
&\rho_\eps =  \eps \rho_1 + \mathcal{O}(\eps^2), \\
&\phi_\eps = \frac{1}{\eps} +  \phi_1 + \eps \phi_2 + \mathcal{O}(\eps^2), \\
&u_\eps = u_1 + \eps u_2 + \mathcal{O}(\eps^2).
\end{aligned}
  \right.
\end{equation}
In particular this means that we expect that $\rho_\eps$ weakly converges to $0$ (which is of course integrable on $\R$). The formal computations then remain the same. In the end, we may obtain the same result as Theorem \ref{KdV} except that we consider the following relative entropy instead of \eqref{RE}:
\begin{equation}
\begin{aligned}
\mathcal{H}_\eps (t) &:= \frac{1}{2} \int_{\R\times\R} f_\eps | v - u_1 - \eps u_2 |^2 \, dv \, dx + \frac{1}{2} \eps \int_\R | \pa_x \phi_\eps - \pa_x \phi_1 - \eps \pa_x \phi_2 |^2 \,dx \\
&+ \frac{1}{2} \int_{\R\times\R}  \left(\phi_\eps - \frac{1}{\eps}- \phi_1 - \eps \phi_2\right)^2 \, dx.
\end{aligned}
\end{equation}


\bigskip
\begin{ack}
The author would like to thank David Lannes for stimulating discussions on the long wave limit during the preparation of this work.
\end{ack}

\bibliographystyle{plain}
\bibliography{vptokdv}

\end{document}